\documentclass[12pt,reqno]{amsproc}
\usepackage{}
\usepackage{mathrsfs, amsfonts, amsthm, amsmath, amssymb}
\usepackage{hyperref}
\usepackage{enumitem}
\usepackage{cleveref}

\hypersetup{
  colorlinks   =  true, 
  urlcolor     = black, 
  linkcolor    = black, 
  citecolor   = black 
}

\usepackage{lineno}

\theoremstyle{plain}
\newtheorem{theorem}{Theorem}[section]

\newtheorem{claim}[theorem]{Claim}

\newtheorem{corollary}[theorem]{Corollary}

\newtheorem{definition}[theorem]{Definition}

\newtheorem{lemma}[theorem]{Lemma}

\newtheorem{proposition}[theorem]{Proposition}

\theoremstyle{remark} 
\newtheorem{example}[theorem]{Example}
\newtheorem{remark}[theorem]{Remark}

\makeatletter
\newcommand{\LeftEqNo}{\let\veqno\@@leqno}

\makeatother

 \numberwithin{equation}  {section}

\def\A{\mathcal A}
\def\B{\mathcal B}

\def\M{\mathcal M}

\def\P{\mathcal P}

\def\R{\mathcal R}
\def\SS{\mathcal S}

\def\CCC{\mathbb C}
\def\FFF{\mathbb F}
\def\NNN{\mathbb N}

\newtheorem*{HalmosProblem}{\bf Problem 8}

\begin{document}

\

\vspace{-2cm}


\title{Reducible operators in non-$\Gamma$ type ${\rm II}_1$ factors} 

\author{Junhao Shen}
\address{Junhao Shen, Department of Mathematics \& Statistics, University of
New Hampshire, Durham, 03824, US}
\email{Junhao.Shen@unh.edu}
\thanks{Junhao Shen was partly supported by TBA}

\author{Rui Shi}
\address{Rui Shi, School of Mathematical Sciences, Dalian University of
Technology, Dalian, 116024, China}
\email{ruishi@dlut.edu.cn, ruishi.math@gmail.com}
\thanks{Rui Shi was partly supported by NSFC(Grant No.12271074).}






\begin{abstract}
The eighth problem of Halmos in \cite{Hal2} asks whether every operator on a separable infinite-dimensional Hilbert space is a norm limit of reducible operators. In \cite{Voi2},  Voiculescu gave this problem an affirmative answer by his remarkable non-commutative Weyl-von Neumann theorem. 

In the paper, we investigate the analogous question for type ${\rm II}_1$ factors. First, we give a characterization of Murray and von Neumann's property $\Gamma$ for a type ${\rm II}_1$ factor in Theorem \ref{prop3.9-gamma}. By this characterization, we answer Problem 2.11 of \cite{Sherman} by proving equivalent formulations of a McDuff factor in Corollary \ref{Sherman}.  Then in Theorem \ref{prop5.3} we develop a spectral gap property for a single operator in a non-$\Gamma$ factor of type ${\rm II}_1$. Based on this spectral gap property,   we prove in Theorem \ref{main-goal} that, {\emph{in the operator norm topology}}, the set of reducible operators is \emph{nowhere} dense in each non-$\Gamma$ factor $\mathcal{M}$ of type ${\rm II}_1$, where \emph{separable} and \emph{non-separable} cases of $\mathcal{M}$ are both considered.
\end{abstract}


\maketitle

\section{Introduction}

Let $\mathcal{H}$ be a complex Hilbert space. Denote by $\mathcal{B}(\mathcal{H})$ the set of all bounded linear operators on $\mathcal{H}$. A \emph{von Neumann algebra} is a self-adjoint subalgebra of $\mathcal B(\mathcal H)$ that is closed in the weak operator topology and contains the identity of $ \mathcal B(\mathcal H) $. A \emph{factor} is a von Neumann algebra whose center consists of scalar multiples of the identity. Factors are further classified by  Murray and von Neumann into type I$_n$,  I$_\infty$,  II$_1$, II$_\infty$, and  III factors (see \cite[Section 6.5]{Kadison2}). By definition, $\mathcal B(\mathcal H)$ is a type I factor.

When $\mathcal H$ is separable, Halmos proved in \cite{Hal} that the set of irreducible operators in $ \mathcal B(\mathcal H) $ is a dense $G_\delta$ subset of $ \mathcal B(\mathcal H) $ in the operator norm topology. Recall that an operator $x\in  \mathcal B(\mathcal H) $ is \emph{reducible} if $x$ has nontrivial reducing closed subspaces. And $x\in  \mathcal B(\mathcal H) $ is \emph{irreducible} if $x$ has no nontrivial reducing closed subspaces, i.e., if $p$ is a projection in $\mathcal{B}(\mathcal{H})$ such that $px = xp$, then $p=0$ or $p=I$.

 Similarly, an element $x$ in a factor $\mathcal N$ is   {\em reducible} if there is a nontrivial projection $p$ in $\mathcal N$ such that $xp = px$. Furthermore, an element $x$ in   $\mathcal N$ is   {\em irreducible} if $x$ is not reducible in $\mathcal N$. Note that a single generator of a factor with separable predual is an irreducible operator.
Thus, in a factor with separable predual, there always exist irreducible operators (see \cite{ Pop, Wogen}).   Based on this, it is natural to consider Halmos' theorem from \cite{Hal} in the setting of factors with separable predual. The authors in \cite{Fang} proved that in each factor $\mathcal{N}$ with separable predual, the set of irreducible operators in $\mathcal{N}$ is operator-norm dense and $G_{\delta}$. 
Later, the author in \cite{Shirui1} proved that in each semifinite factor $\mathcal{N}$ with separable predual,  the set of irreducible operators in $\mathcal{N}$ is dense with respect to $\mathrm{max}\{ \Vert \cdot \Vert, \Vert \cdot \Vert_p \}$-norm for $p>1$, where the $\mathrm{max}\{ \Vert \cdot \Vert, \Vert \cdot \Vert_p \}$-norm works as an analogue of the Schatten $p$-norm on the set of finite-rank operators in $ \mathcal B(\mathcal H) $.

On the other hand, the eighth problem in \cite{Hal2} raised by Halmos is stated as follows.

\begin{HalmosProblem}\label{8th}
    On a separable Hilbert space, is every operator the norm limit of reducible ones?
\end{HalmosProblem}

On a finite-dimensional Hilbert space, the answer to the problem is negative, since the set of reducible operators is closed and nowhere dense in the operator norm topology (see \cite[Main theorem]{Hal} and \cite[p.919]{Hal2}). On a separable, infinite-dimensional Hilbert space, the problem was answered affirmatively by Voiculescu as an application of his non-commutative Weyl-von Neumann theorem in \cite{Voi2}.

Inspired by some recent research on irreducible operators in factors \cite{Fang}, normal operators in semi-finite factors \cite{Hadwin3,Li,Li2}, and similar operator-theoretic results for finite von Neumann algebras \cite{CDSZ,DNZ}, we investigate {\bf Problem 8} in type ${\rm II}_1$ factors in the current paper.

Let $\mathcal M$ be a factor of type ${\rm II}_1$ with    trace $\tau$. Through out the paper, we denote by  $\|\cdot\|$   the operator norm on $\mathcal M$, and by $\|\cdot\|_2$ the $2$-norm  on $\mathcal M$, i.e.,  $\|x\|_2=\sqrt {\tau (x^*x)}$ for all $x\in \mathcal M$. For elements $x$ and $y$ in $\mathcal M$, we denote by $[x,y]=xy-yx$ the \emph{commutator} of $x$ and $y$.

We will frequently mention Murray and von Neumann's property $\Gamma$ for type ${\rm II}_1$ factors (see \cite[Definition 6.1.1]{MvN}), since we 
develop techniques from von Neumann algebras to answer Halmos' {\bf Problem 8} negatively in type ${\rm II}_1$ factors without property $\Gamma$. 
Historically, 
property $\Gamma$ is the first invariant used by Murray and von Neumann in \cite{MvN}  to show the existence of non-hyperfinite type ${\rm II}_1$ factors, and it plays a critical role in Connes' celebrated paper \cite{Connes}.  

Recall that a type ${\rm II}_1$ factor $\mathcal{M}$ has \emph{property $\Gamma$ } if and only if {\em  for any finitely many elements $x_1,\ldots, x_n$   in $\mathcal M$ and  any $\varepsilon>0$, there exists a unitary element $u$ in $\mathcal M$ with $\tau(u)=0$ such that $\|[x_j, u]\|_2\le \varepsilon$ for all $1\le j\le n$.}   Notice that the original definition of property $\Gamma$ for a type ${\rm II}_1$ factor $\mathcal M$  does not require $\mathcal M$   having separable predual.
For simplicity, for  a type ${\rm II}_1$ factor $\mathcal M$ without property $\Gamma$, we say that $\mathcal M$ is \emph{{non-$\Gamma$}} (When $\mathcal M$ has separable predual, $\mathcal M$ is non-$\Gamma$ if and only if $\mathcal M$ is full \cite{Co74}).

The main purpose of this paper is to prove the following theorem.

\vspace{0.2cm} 
{T{\scriptsize HEOREM}} \ref{main-goal}.
\emph{Let $\mathcal{M}$ be a non-$\Gamma$ type $ {\rm II}_1$ factor. Then, in the operator norm topology, the set of reducible operators in $\mathcal M$ is nowhere dense and not closed in $\mathcal M$. }

\vspace{0.1in}

To prove Theorem \ref{main-goal}, we take four steps and prepare \Cref{thm3.1}, \Cref{prop5.3}, and \Cref{may14Thm5.14}, which are also of independent interest. 

\vspace{0.1in}

\noindent{\textbf{Step One.}} Inspired by Dixmier's ideas in \cite{Dixm}, we develop another characterization of property $\Gamma$ for type $ {\rm II}_1$ factors.

\vspace{0.2cm}
 {P{\scriptsize ROPOSITION}} \ref{prop3.9-gamma}. \emph{Let $\mathcal M$ be
a type ${\rm II}_1$ factor with   trace $\tau$. Then the following
statements are equivalent:
    \begin{enumerate}[label={\rm(\roman*)}]
        \item   $\mathcal M$ has property $\Gamma$ of Murray and von Neumann.
        \item   For every $x$ in $\mathcal M$, $W^*(x)'\cap \mathcal M^\omega $ is  diffuse.
        \item   For every $x$ in $\mathcal M$ and every nonzero projection $p$ in $\mathcal M$, $$W^*(p x p)'\cap (p\mathcal M p)^\omega \ne \mathbb C p. $$
        \item    For every $x $ in $\mathcal M$,  $W^*(x )'\cap \mathcal M^\omega \ne \mathbb C I$.
  \end{enumerate}
Here $W^*(x)$ is the von Neumann subalgebra generated by $x$ in $\mathcal M$ and $\omega$ is a free ultrafilter on the set $\mathbb N$ of natural numbers.}

\vspace{0.1in}

As a direct application of \Cref{prop3.9-gamma}, we prove a new characterization of property $\Gamma$ for type ${\rm II}_1$ factors in \Cref{thm3.1}.

\vspace{0.2cm} {T{\scriptsize HEOREM}} \ref{thm3.1}
\emph{   Let $\mathcal M$ be a type ${\rm II}_1$ factor with    trace $\tau$. Then $\mathcal M$ has property $\Gamma$ if and only if, for any element $x$  in $\mathcal M$ and any $\varepsilon>0$, there exists a unitary element $u$ in $\mathcal M$ such that $\tau(u)=0$ and $\|[x, u]\|_2\le \varepsilon$.}

\vspace{0.1in}

To illustrate the interest of \Cref{thm3.1}, it is worth remarking that in \cite[Section 3.2.2]{FHS} the authors showed that property $\Gamma$ can be axiomatized by a sequence of sentences $\{\sigma_n\}_{n\ge 1}$.  \Cref{thm3.1} shows that only the first sentence $\sigma_1$ is needed.

\vspace{0.1in}

\noindent{\textbf{Step Two.}} By Proposition \ref{prop3.9-gamma} and a lemma by Connes in \cite{Connes}, we obtain a single operator with spectral gap in each non-$\Gamma$ type ${\rm II}_1$ factor. 

\vspace{0.2cm} {T{\scriptsize HEOREM}} \ref{prop5.3}.
\emph{Let $\mathcal M$ be a {non-$\Gamma$} type ${\rm II}_1$  factor with   trace $\tau$. Then there exist two self-adjoint elements $x_1$, $x_2$ in $\mathcal M$ and a positive number $\alpha>0$ such that
  \begin{equation*}
    \|[x_1, e]\|_2+\|[x_2, e]\|_2\ge \alpha \|e\|_2\|I-e\|_2, \quad \text{ for every projection } \, e \in  \mathcal M
    .
  \end{equation*}
}

\vspace{0.1in}

Combining \Cref{prop5.3} and Marrakchi's Proposition 2.2 of \cite{Ma}, we obtain an operator-theoretic characterization of a {non-$\Gamma$} type ${\rm II}_1$ factor.

\vspace{0.2cm} {C{\scriptsize OROLLARY}} \ref{may17Cor4.12}.
\emph{
    Let $\mathcal M$ be a type ${\rm II}_1$ factor with   trace $\tau$. Then the following statements are equivalent:
    \begin{enumerate}[label={\rm (\roman*)}]
    \item   $\mathcal M$ is {non-$\Gamma$}, i.e., $\mathcal M$ fails to have property $\Gamma$.
    \item   There exist  two self-adjoint elements   $x_1$ and $x_2$ in $\mathcal M$ and an $\alpha_1>0$ such that
    \begin{equation*}
        \|[x_1,y ]\|_2+ \|[x_2,y ]\|_2 \ge   \alpha_1 \|y-\tau(y)\|_2, \quad \text { for every } y\in \mathcal
        M.
    \end{equation*}
    \item   There exist an $x$ in $\mathcal M$  and an $\alpha_2>0$  such that
    \begin{equation*}
        \|[x, y]\|_2+\|[x^*, y]\|_2\ge \alpha_2 \|y-\tau(y)\|_2, \qquad \text{ for every } y \in\mathcal
        M.
    \end{equation*}
    \item   There exist an $x$ in $\mathcal M$  and an $\alpha_3>0$  such that
    \begin{equation*}
        \|[x, y]\|_2\ge \alpha_3 \|y-\tau(y)\|_2, \qquad \text{ for every self-adjoint } y \in\mathcal
        M.
    \end{equation*}
    \item   There exist two unitary elements $u_1$ and $u_2$ in $\mathcal M$  and an $\alpha_4>0$  such that
    \begin{equation*}
        \|[u_1, y]\|_2 +\|[u_2, y]\|_2  \ge \alpha_4 \|y-\tau(y)\|_2, \qquad \text{ for every } y \in\mathcal
        M.
    \end{equation*}
    \end{enumerate}
}

\vspace{0.1in}

From Theorem 2.1 (c) of \cite{Connes}, in a type ${\rm II}_1$ factor $(\mathcal{M},\tau)$ without property $\Gamma$, there exist finitely many unitary operators $u_1,\ldots,u_k$ in $\mathcal{M}$ such that for each operator $y$ in $\mathcal{M}$ the averaging mapping $y \to u_1 y u_1^*+\cdots+u_k y u_k^*$ acts with spectral gap on the Hilbert space $L^2\mathcal{M},\tau)$. Precisely, there exist a finite subset $\{u_1,\ldots,u_k\}$ in $\mathcal{U}(\mathcal{M})$, the set of unitary operators in $\mathcal{M}$, and  a number $c>0$ such that
\begin{equation*}
    \sum_{1\le j \le k} \|[u_j, y]\|_2 \ge c \|y-\tau(y)\|_2, \quad \text{ for every } y \in\mathcal{M}.
\end{equation*}
By \Cref{may17Cor4.12}, one can always take $k=2$.

\vspace{0.1in}

\noindent{\textbf{Step Three.}}  A key observation, connecting the spectral gap property of an operator and the operator norm closure of reducible operators, is the following lemma.

\vspace{0.2cm} {L{\scriptsize EMMA}} \ref{may15Lemma5.6}
\emph{  Let $x_1$ and $x_2$ be self-adjoint elements in $\mathcal M$. If there exist a positive number $\alpha>0$ and a projection $p\in \mathcal M$ with $\tau(p)>0$, satisfying
    \begin{equation*}
        \|[x_1,p]\|_2+ \|[x_2,p]\|_2\ge \alpha \|p\|_2,
    \end{equation*}
    then
    \begin{equation*}
        \|[x_1,p]\| + \|[x_2,p]\| \ge \frac \alpha  {\sqrt 2}.
    \end{equation*} }

\noindent  Now the existence of elements that are not contained in the operator norm closure of reducible operators, in a non-$\Gamma$ type $\mathrm{II}_1$ factor is a combination of Theorem \ref{prop5.3} and Lemma \ref{may15Lemma5.6}.

\vspace{0.2cm} {T{\scriptsize HEOREM}} \ref{may14Thm5.14}
\emph{Let $\mathcal M$ be a non-$\Gamma$ type ${\rm II}_1$ factor. Then $\overline{\operatorname{Red}(\mathcal M)}^{\|\cdot\|}\ne \mathcal M$, where $\overline{\operatorname{Red}(\mathcal M)}^{\|\cdot\|}$ is the operator norm closure of  $\operatorname{Red}(\mathcal M)$ the set of reducible operators in $\mathcal M$.
}

\vspace{0.1in}

\noindent{\textbf{Step Four.}} Finally, based on Theorem \ref{may14Thm5.14},  Theorem  \ref{main-goal} is proved.

\vspace{0.1in}

We mention that there are nonseparable type ${\rm II}_1$ factors with property $\Gamma$ in which all operators are reducible  (see Example \ref{June10Example5.2}). On the other hand, in Proposition \ref{June28Prop5.2} we prove that, if a type ${\rm II}_1$ factor has separable predual, then the set of reducible operators is not closed in the operator norm topology.  It remains an open question whether, in a type ${\rm II}_1$ factor with separable predual and with property $\Gamma$, the set of reducible operators is dense in the operator norm topology.

This paper is organized as follows. In Section 2,  we recall
ultrapower algebras, central sequence algebras, and property
$\Gamma$ for type ${\rm II}_1$ factors. Some useful techniques are
also prepared. In Section 3, in the view of a single operator, we
prove a characterization of property $\Gamma$ for type ${\rm II}_1$
factors in Proposition \ref{prop3.9-gamma}. As an application of
Proposition \ref{prop3.9-gamma}, we provide an answer to Sherman's
question in Problem 2.11 of \cite{Sherman}, where we show equivalent
characterizations of McDuff factors. In Section 4,  the existence of a single operator with   spectral gap
in a non-$\Gamma$ type ${\rm II}_1$ factor  is shown in Theorem
\ref{prop5.3}. In Section 5, we prove in Theorem \ref{may14Thm5.14}
that reducible operators are not dense in non-$\Gamma$ type ${\rm
II}_1$ factors. In Section 6, we further prove in Theorem
\ref{main-goal} that reducible operators are nowhere dense in
non-$\Gamma$ type ${\rm II}_1$ factors with the techniques developed
in the preceding sections.

\section{Preliminaries and Notation}

In this section, we recall some background on ultrapowers and central sequence algebras for type ${\rm II}_1$ factors, which will be used in proving our characterization of property $\Gamma$.


Let $\mathcal M$ be a type ${\rm II}_1$  factor  with   trace $\tau$.  Let $\mathbb N$ be the set of all the natural numbers and $\omega\in\beta \mathbb N\setminus \mathbb N$  a fixed free ultrafilter over $\mathbb N$. Let
\begin{equation*}
    \ell^{\infty} ( \mathcal M) =\{(a_n)_n \ : \ \forall \ n\in \mathbb N, \  a_n\in\mathcal M \ \text { and } \ \sup_{n\in\mathbb N}
    \|a_n\|<\infty\},
\end{equation*}
and
\begin{equation*}
    \mathcal I_\omega( \mathcal M) =\{  (a_n)_n \in   \ell^{\infty} ( \mathcal M) \ : \ \lim_{n\rightarrow \omega } \|a_n\|_2
    =0\}.
\end{equation*}
Then $\mathcal I_\omega( \mathcal M) $ is a two sided ideal of $ \ell^{\infty} ( \mathcal M)$ and the ultrapower of $\mathcal M$ along $\omega$, denoted by $\M^{\omega}$, is defined to be
\begin{equation*}
    \mathcal M^\omega  = \ell^{\infty} ( \mathcal M)/ \mathcal I_\omega( \mathcal
    M).
\end{equation*}
If no confusion arises, an element in  $\mathcal M^\omega$ will be
denoted by $(a_n)_\omega$.  By \cite{Wright} or  \cite{Sak}, $\mathcal M^\omega$ is a type ${\rm II}_1$ factor with a
natural trace $\tau_\omega$ (also see    Theorem A.3.5 in
\cite{SS}).  If $\mathcal P$ is a von Neumann
subalgebra of $\mathcal M$, then we view $\mathcal P^\omega
\subseteq \mathcal M^\omega$ (see the discussion after Definition
A.4.1. in \cite{SS}). Moreover, there is a natural embedding from
$\mathcal M$ into $\mathcal M^\omega$ by sending each $a\in \mathcal
M$ to a constant sequence $(a, a, a,\ldots)_\omega$   in $\mathcal
M^\omega$. Thus, we view $\mathcal M\subseteq \mathcal M^\omega$. Let $\mathcal M_\omega = \mathcal M^{\prime} \cap \mathcal M^\omega$, which is called the \emph{central sequence algebra} of $\mathcal M^\omega$.

In the following lemma, item \ref{lemma2.1.i} is  Lemma A.5.5 of \cite{SS}. Item \ref{lemma2.1.ii} follows from \ref{lemma2.1.i}.

\begin{lemma}[Lemma A.5.5 of \cite{SS}]\label{lemma1}
    Let $\mathcal P$ be a type ${\rm II}_1$ factor with trace $\tau$ and $\mathcal M_k(\mathbb C)$ is a matrix algebra of size $k\in\mathbb N$. 
    \begin{enumerate}[label={\rm(\roman*)}]
        \item \label{lemma2.1.i}  $\displaystyle (\mathcal P\otimes \mathcal M_k(\mathbb C) )^\omega = \mathcal P^\omega \otimes \mathcal M_k(\mathbb C)$;
        \item \label{lemma2.1.ii}    If $\mathcal A$ is a von Neumann subalgebra of $\mathcal P$, then
        \begin{equation*}
            \displaystyle (\mathcal A\otimes \mathcal M_k(\mathbb C) )' \cap (\mathcal P\otimes
            \mathcal M_k(\mathbb C) )^\omega = ( \mathcal A'\cap \mathcal P^\omega )\otimes
            1_k,
    \end{equation*} 
    where $1_k$ is the identity of $\mathcal M_k(\mathbb C)$.
    \end{enumerate}
\end{lemma}

Recall that a type ${\rm II}_1$ factor $\mathcal{M}$ has property $\Gamma$ of Murray and von Neumann \cite{MvN} if and only if {\em  for any family $x_1,\ldots, x_n$ of finitely many elements in $\mathcal M$ and  any $\varepsilon>0$, there exists a unitary element $u$ in $\mathcal M$ with $\tau(u)=0$ such that $\|[x_i, u]\|_2\le \varepsilon$ for all $1\le i\le n$.  }

\begin{definition}
    For $x\in \mathcal M$ and $\mathcal S\subseteq \mathcal M$, we define a distance function as follows:
    \begin{equation*}
        \operatorname{dist}_{\|\cdot\|_2}(x, \mathcal S) =\inf \{\|x-y\|_2   : y\in\mathcal
        S\}.
    \end{equation*}
    Denote by $W^*(\mathcal S)$ the von Neumann subalgebra generated by $\mathcal S$ in $\mathcal M$ and by $C^*(\mathcal S)$ the unital $C^*$-subalgebra generated by $\mathcal S$ in $\mathcal M$. The relative commutant of $W^*(\mathcal S)$ in $\mathcal M$ is denoted by $W^*(\mathcal S)'\cap\mathcal M$.
\end{definition}

\begin{remark}
    It is obvious that if a  type ${\rm II}_1$ factor $\mathcal M$ has property $\Gamma$, then
    \begin{equation*}
     W^*(x_1,\ldots, x_n)'\cap \mathcal M^\omega\ne \mathbb CI
    \end{equation*}
    for any finitely many elements $x_1,\ldots, x_n$ in $\mathcal M$. If $\mathcal M$ has   separable predual, then $\mathcal M$ has property $\Gamma$ if and only if $\mathcal M'\cap \mathcal M^\omega \ne \mathbb CI$ {\rm (see \cite {Dixm})}.    It is worthwhile noting that there exist examples of (nonseparable) type ${\rm II}_1$ factors $\mathcal M$ with property $\Gamma$ such that $\mathcal M'\cap \mathcal M^\omega=\mathbb CI$ {\rm (see Proposition $3.7$ in \cite {FHS})}.
\end{remark}

Examples of (nonseparable) type ${\rm II}_1$ factors $\mathcal M$ with property $\Gamma$ and with $\mathcal M'\cap\mathcal M^\omega=\mathbb CI$ can also be found in the following proposition, which is a consequence of Popa's result in \cite{Popa2}.

\begin{proposition}
    Let $\mathcal N$ be a type ${\rm II}_1$ factor with separable predual and with property $\Gamma$. Let $\omega$  be a free ultrafilter on $\mathbb N$ and $\mathcal M=\mathcal N^\omega$ an ultrapower of $\mathcal N$ along   $\omega$. Then $\mathcal M$ is a type ${\rm II}_1$ factor  with property $\Gamma$ and with $\mathcal M'\cap \mathcal M^\omega=\mathbb CI$.
\end{proposition}

\begin{proof}
    It is known that $\mathcal M=\mathcal N^\omega$ is a type ${\rm II}_1$ factor (see Theorem A.3.5 in \cite {SS}). It is straightforward to verify that $\mathcal M$ has property $\Gamma$. We need only to show that $\mathcal M'\cap \mathcal M^\omega=\mathbb CI$.

    The traces on $\mathcal N$, $\mathcal M$, and $\mathcal M^\omega$ will be  denoted by $\tau_{\mathcal N}$, $\tau_{\mathcal M}$, and $\tau_{\mathcal M^\omega}$ respectively. The 2-norms induced by the corresponding traces on $\mathcal N$, $\mathcal M$
     and $\mathcal M^\omega$  will be  denoted by $ \|\cdot\|_{2,\mathcal N}$, $ \|\cdot\|_{2,\mathcal M}$
     and $ \|\cdot\|_{2,\mathcal M^\omega}$ respectively. Elements in  $\mathcal N$, $\mathcal M$ and $\mathcal M^\omega$  will be
      denoted by $x$, $X$ or $(x_n)_\omega$, and $(X_m)_\omega$ respectively if there is no confusion.

    Suppose that  $\mathcal M'\cap \mathcal M^\omega\ne \mathbb CI$ and $(P_m)_\omega$ is a nontrivial projection in $\mathcal M'\cap \mathcal M^\omega$. Let $\lambda = \tau_{\mathcal M^\omega}((P_m)_\omega)$. Then $0<\lambda<1$. By Theorem A.5.3 in \cite{SS}, we assume that each  $P_m$ is a projection in $\mathcal M$ with $\tau_{\mathcal M}(P_m)=\lambda $ for   $m\ge 1$. As $P_m=(p_n^{(m)})_\omega$ is in $\mathcal M=\mathcal N^\omega$, we further assume that each $p_n^{(m)}$ is a projection in $\mathcal N$ with $\tau_{\mathcal N}(p_n^{(m)})=\lambda$ for $n,m\ge 1$.

    By Corollary on page 187 in \cite{Popa2}, there exists a family $\{u_n\}_{n=1}^\infty$ of unitary elements in $\mathcal N$ such that $\tau_{\mathcal N}(u_n)=0$ and
    \begin{equation}
        \lim_{n\rightarrow \infty}
        \tau_{\mathcal N}(u_nb_1u_n^*b_2)=\tau_{\mathcal N}(b_1)\tau_{\mathcal N}(b_2), \qquad \forall \ b_1,b_2\in\mathcal N.\label{Jun20Equa2.1}
    \end{equation}
    For each $n\ge 1$, by Equation (\ref{Jun20Equa2.1}), we let $k_n$ be a positive integer such that
    \begin{equation*}
        |\tau_{\mathcal N}(u_{k_n}p_n^{(m)}u_{k_n}^*p_n^{(m)}) -(\tau_{\mathcal N}(p_n^{(m)}))^2 |= |\tau_{\mathcal N}(u_{k_n}p_n^{(m)}u_{k_n}^*p_n^{(m)}) -\lambda^2|\le 1 /n, \ \ \forall \ 1\le m\le
        n.
    \end{equation*}
    So, for $1\le m\le n$,
    \begin{equation*}
        | \|[u_{k_n}, p_n^{(m)}]\|_{2,\mathcal N}^2 -(2\lambda -2\lambda^2)|= |2\tau_{\mathcal N} (p_n^{(m)})-2\tau_{\mathcal N}(u_{k_n}p_n^{(m)}u_{k_n}^*p_n^{(m)}) -(2\lambda -2\lambda^2)| \le
        2/n.
    \end{equation*}
    Let $V=(u_{k_n})_\omega$ be a unitary element in $\mathcal M=\mathcal N^\omega$. Then
    \begin{equation*}
        \|[V, P_m]\|_{2,\mathcal M}^2 =\lim_{n\rightarrow \omega}\|[u_{k_n}, p_n^{(m)}]\|_{2,\mathcal N}^2=2\lambda -2\lambda^2 >0, \ \ \forall \ m\ge 1.
            \end{equation*}
    This contradicts the assumption that $(P_m)_\omega$ is  in $\mathcal M'\cap \mathcal M^\omega$. Hence $\mathcal M'\cap \mathcal M^\omega=\mathbb CI$.
\end{proof}

The next lemma is well-known. We include its proof for completeness.

\begin{lemma}\label{lemma2.2}
    Let $\mathcal M$ be a type ${\rm II}_1$ factor with    trace $\tau$. Suppose that $p$ is a nonzero projection in $\mathcal M$. Then $\mathcal M$ has property $\Gamma$ if and only if $p\mathcal Mp$ has property $\Gamma$.
\end{lemma}

\begin{proof} When $\mathcal M$ has separable predual,   the result can be found in Proposition 1.11 of \cite{PimPopa}. Now we assume that $\mathcal M$ has nonseparable predual.

      (i). Suppose that  $\mathcal{M}$ has property $\Gamma$. Let $x_1,\ldots, x_n$ be   in $p\mathcal M p$ and $\varepsilon>0$.  By Proposition 7.1 in \cite{CPSS}, there exists a  subfactor $\mathcal M_1$, with separable predual and with property $\Gamma$, such that $ \{p, x_1,\ldots, x_n, I\}\subseteq  \mathcal M_1\subseteq \mathcal M$. Then $p \M_1 p$ has property $\Gamma$ by Lemma 2.5 of \cite{Connes}. So there exists a unitary $u$ in $p\mathcal M_1p\subseteq p\mathcal Mp$, with $\tau(u)=0$, such that $\|[x_i, u]\|_2\le \varepsilon$ for $1\le i\le n$. By definition, $p\mathcal M p$ has property $\Gamma$.

          (ii). Assume that $p\mathcal M p$ has property $\Gamma$. Let $n\in\mathbb N$ and $q$ be a subprojection of $p$ such that $\tau(q)=1/n$. By part (i), $q\mathcal M q$ has property $\Gamma$. Notice that $\mathcal{M}$ is $\ast$-isomorphic to the von Neumann algebra tensor product $q\mathcal M q \otimes \mathcal M_n(\mathbb C)$, which is denoted by $\mathcal M\cong q\mathcal M q \otimes \mathcal M_n(\mathbb C)$. From Theorem 13.4.5 of \cite{SS}, it follows that $\mathcal M$ has property $\Gamma$.
\end{proof}

A  quick consequence of spectral theory is needed in the paper and its proof is sketched.

\begin{lemma}\label{lemma5.2}
    Let $\{p_i\}_{i=1}^n$ be a family of mutually orthogonal projections in $\mathcal M$ such that $p_1+ \cdots + p_n=I$. Suppose that $\{x_i\}_{i=1}^n$ is a family of elements in $\mathcal M$ satisfying
    \begin{enumerate}[label={\rm(\roman*)}]
        \item   $x_i$ is in $p_i\mathcal Mp_i$ for each $1\le i \le n$,
        \item   as an operator in  $p_i\mathcal Mp_i$,    $x_i$ is self-adjoint and invertible for each $1\le i\le n$, and
        \item   $\sigma_{p_i\mathcal Mp_i}(x_i) \cap \sigma_{p_j\mathcal Mp_j}(x_j)=\emptyset$, $\forall \ 1\le i\ne j\le n$, i.e., the spectra of $x_i$ and $x_j$ are pairwise disjoint for $i\neq j$.
    \end{enumerate}
    If
    $
        x=x_1+ x_2+ \cdots + x_n,
   $
    then
    \begin{equation*}
        \{p_1,\ldots, p_n, x_1,\ldots, x_n\}\subseteq  C^* (x)\subseteq W^*(x).
    \end{equation*}
\end{lemma}

\begin{proof}
    Since $x_i$ is self-adjoint for $1\leq i \leq n$, it follows that $x$ is self-adjoint. Note that the spectra of $x_i$ and $x_j$ are pairwise disjoint for $1\leq i\neq j \leq n$. Define continuous functions $f_{i}(t)$ on $\sigma( x)$ as follows:
    \begin{equation*}
        f_{i}(t)=\begin{cases}
            1,& t\in \sigma_{p_i\mathcal Mp_i}(x_i);\\
            0,& t\in \sigma( x)\backslash \sigma_{p_i\mathcal Mp_i}(x_i).
        \end{cases}
    \end{equation*}
    That $x_i$ is invertible in  $p_i\mathcal Mp_i$ entails $p_i=f_i(x)\in  C^* (x)$ for $1\leq i \leq n$. Moreover, $x_i=p_i x p_i$ belongs to $ C^* (x)$ for $1\leq i \leq n$. This completes the proof.
\end{proof}

\section{A characterization of property \texorpdfstring{$\Gamma$} \ \ for type \texorpdfstring{${\rm II}_1$}\ \   factors}

Let $\mathcal M$ be a type ${\rm II}_1$ factor with trace $\tau$. It is an open question whether a type ${\rm II}_1$ factor with separable predual is generated by a single operator. When $\mathcal M$ is singly generated,  the following Theorem \ref{thm3.1} is a direct consequence of the definition of property $\Gamma$.  The main goal of this section is to provide an equivalent characterization of property $\Gamma$ for type ${\rm II}_1$ factors without the assumption on the cardinality of generators.

\begin{theorem} \label{thm3.1}
    Let $\mathcal M$ be a type ${\rm II}_1$ factor with    trace $\tau$. Then $\mathcal M$ has property $\Gamma$ if and only if, for any element $x$  in $\mathcal M$ and any $\varepsilon>0$, there exists a unitary element $u$ in $\mathcal M$ such that $\tau(u)=0$ and $\|[x, u]\|_2\le \varepsilon$.
\end{theorem}

The proof of Theorem \ref{thm3.1} is postponed until after a few technical lemmas. We start with a definition of the support of an operator with respect to a family of mutually orthogonal projections.

\begin{definition} \label{def3}
    Let $x\in \M$ and $\{p_i\}_{i=1}^k\subseteq\M$ be a family of mutually orthogonal equivalent projections with $p_1+\cdots + p_k=I$. Define
    \begin{equation*}
        \operatorname{supp}(x, \{p_i\}^{k}_{i=1}) = \bigvee \{p_i : p_ix\ne 0 \ \text { or }  \ xp_i \ne 0, \ 1\leq i \leq k
        \}.
    \end{equation*}
\end{definition}

\begin{lemma}\label{lemma3.4}
    Let $k\in\mathbb N$. Suppose that $\mathcal A$ is a type ${\rm I}_k$ subfactor of $\mathcal M$  and $\{e_{ij}\}_{i,j=1}^k$ is a system of matrix units of $\mathcal A$. Let $q_1$ be a projection in $\mathcal M$.

    If $\{i_1,\ldots, i_{\ell}\}\subseteq \{1,\ldots, k\}$ with $\ell \ge k\cdot\tau(\operatorname{supp}(q_1, \{e_{ii}\}_{i=1}^k))$, then there exists a projection  $q\in \mathcal M$ such that
    \begin{enumerate}[label={\rm(\roman*)}]
        \item   $\operatorname{supp}(q, \{e_{ii}\}_{i=1}^k) \le  e_{i_1i_1}+\cdots + e_{i_{\ell} i_{\ell}}$ and
        \item   $W^*(q_1, \mathcal A)= W^*(q , \mathcal A)$.
    \end{enumerate}
\end{lemma}

\begin{proof}
List $ \{e_{ii}   :   e_{ii}q_1\ne 0 \text { or } q_1e_{ii}\ne 0  \ \mbox{ for }\ 1\leq i\leq k\}$ as $\{e_{j_1j_1},\ldots, e_{j_mj_m}\} $, where $m= k\cdot\tau(\operatorname{supp}(q_1, \{e_{ii}\}_{i=1}^k)) $ is an integer by Definition \ref{def3}. As  $\tau(\operatorname{supp}(q_1, \{e_{ii}\}_{i=1}^k)) \le \ell/k$, we have $m\le \ell$. Thus, from the fact that $\mathcal A$ is a type ${\rm I}_k$ factor, we deduce that  there is a unitary element $u\in \mathcal A$ such that $ue_{j_nj_n}u^*=e_{i_ni_n}$ for $1\le n\le m.$ Now $q=uq_1u^*$ is a desired projection in $\mathcal M$.
\end{proof}

Recall that a von Neumann algebra is called {\em diffuse} if it contains no nonzero minimal projections.
The following lemma is prepared for an induction argument of Claim \ref{Gamma-2}.1. 

\begin{lemma} \label{lemma6}
    Let $k\in\mathbb N$. Suppose that $\mathcal A$ is a type ${\rm I}_k$ subfactor of $\mathcal M$  and $\{e_{ij}\}_{i,j=1}^k$ is a system of matrix units of $\mathcal A$. Assume that $x$ is an element in $\mathcal M$ such that
    $$\begin{aligned}
        W^*(x, \mathcal A)'\cap \mathcal M^\omega \quad \text{ is
        diffuse}.
    \end{aligned}$$
    Then, for any $\varepsilon>0$, there exist a positive integer $m$, a type ${\rm I}_m$ subfactor $\mathcal N $ of $\mathcal A'\cap \mathcal M$, a system of matrix units $\{f_{ij}\}_{i,j=1}^m$ of $\mathcal N $  and a projection $q$ in $\mathcal M$ such that
    \begin{enumerate}[label={\rm (\roman*)}]
      \item  \label{lemma3.4.i}   $q = q (\sum_{i=2}^m f_{ii} e_{11})= (\sum_{i=2}^m f_{ii} e_{11})q$;
      \item  \label{lemma3.4.ii}  $\operatorname{dist}_{\|\cdot\|_2} (x, W^*(q, \mathcal A,\mathcal N ))\le \varepsilon$.
    \end{enumerate}
\end{lemma}

\begin{proof}
    Fix $\varepsilon>0$. Let $m=64k^2$ and $\mathcal P=\mathcal A'\cap \mathcal M$. As $\mathcal A$ is a type ${\rm I}_k$ subfactor of $\mathcal M$, we identify $\mathcal M$ with $\mathcal P\otimes \mathcal A$. 
    There exists a family $\{(p^{(i)}_n)_\omega : 1\le i\le m\}$ of mutually orthogonal projections in $W^*(x, \mathcal A)'\cap \mathcal M^\omega$ with the same trace such that $\sum_{1\le i\le m} (p^{(i)}_n)_\omega = I$.

    By Lemma \ref{lemma1}, $W^*(x, \mathcal A)'\cap \mathcal M^\omega \subseteq \mathcal A'\cap \mathcal M^\omega =\mathcal P^\omega, $ whence $\{(p^{(i)}_n)_\omega : 1\le i\le m\} \subseteq \mathcal P^\omega$.  By Lemma A.5.3 in \cite{SS}, we can further assume that  $\{p_n^{(i)}\}_{1\le i\le m}$ is a family of mutually  orthogonal equivalent projections in  $\mathcal P$ such that $\sum_{1\le i\le m}  p^{(i)}_n  =I$. Thus $\{(p^{(i)}_n)_\omega : 1\le i\le m\}\subseteq  W^*(x, \mathcal A)'\cap \mathcal M^\omega$ implies that there exists a family $\{p_{i } \}_{ 1\le i\le m } $  of mutually orthogonal equivalent  projections in $\mathcal P$ such that
    \begin{enumerate}[label=(\alph*)]
        \item  \label{lemma3.4.a} $p_1+p_2+\cdots + p_m = I$;
        \item  \label{lemma3.4.b} $\tau(p_i)=1/m $ for each $1\le i\le m$;
        \item \label{lemma3.4.c}  $\|x- \sum_{1\le i\le m} p_i x p_i \|_2 \le \varepsilon$.
    \end{enumerate}
    Since $\mathcal P$ is a subfactor, there exist a type ${\rm I}_m$ subfactor $\mathcal N $ of $\mathcal P$  and a system of matrix units $\{f_{ij}\}_{i,j=1}^m$ of $\mathcal N $ such that $f_{ii}=p_i$ for each $1\le i\le m$.

    Define $y= \sum_{i=1}^m p_i x p_i= \sum_{i=1}^m f_{ii} x f_{ii} $. Recall that $\{e_{st}\}^{k}_{s,t=1}$ is a system of matrix units of $\mathcal A$. As $\mathcal A$ commutes with $\mathcal N $, it follows that $W^*(\mathcal A,\mathcal N )$ is a subfactor of type ${\rm I}_{km}$. Moreover, we have that $\{e_{st}f_{ij}\}_{1\le s,t\le k; 1\le i,j\le m}$ is a system of matrix units of $W^*(\mathcal A,\mathcal N )$ and $\{e_{ss}f_{ii}\}_{1\le s\le k; 1\le i\le m}$ is a family of mutually orthogonal projections in $\mathcal M$ such that $\tau(e_{ss}f_{ii})= \frac 1{km}$ for each $1\le s\le k$ and $1\le i\le m$.  It follows that
    \begin{equation*}
        \begin{aligned}
            \sum_{i=1}^m f_{ii} x f_{ii}=y   = \sum_{j=1}^m f_{jj} y f_{jj} 
             = \sum_{s,t=1}^k e_{ss} ( \sum_{i=1}^m  f_{ii} y f_{ii} )e_{tt}  =\sum_{s,t=1}^k\sum_{i=1}^m  e_{ss}   f_{ii} y   e_{tt}f_{ii}.
        \end{aligned}
    \end{equation*}
    When no confusion can arise, we write $|\SS |$ for the cardinality of a set $\SS$. Thus, we obtain the following inequality:
    \begin{equation*}
         \frac {  |\{ (e_{ss}f_{ii}, e_{tt}f_{jj} )  :   e_{ss}   f_{ii} y  e_{tt}   f_{jj}   \ne 0 \ \mbox{ for }   1\leq s,t \leq k  \mbox{ and } 1\leq i, j \leq m \}| }{(km)^2} \leq  \frac {k^2 m}{k^2m^2}=\frac 1
         m.
    \end{equation*}
    By the cut-and-paste theorem (Theorem 4.1 of \cite{Shen}), there exists a projection $q$ in $\mathcal M$ such that
    \begin{enumerate}[resume, label=(\alph*)]
       \item \label{lemma3.4.d} $W^*(y, \mathcal A, \mathcal N ) = W^*(q, \mathcal A, \mathcal N )$;
       \item \label{lemma3.4.e}  $\displaystyle \tau \Big(\operatorname{supp} (q, \{e_{ss}f_{ii}\}_{1\le s\le k; 1\le i\le m}) \Big)\le
       2\left (\frac {1}{ m }\right )^{1/2}+\frac 2 {km} =   2 \left (\frac 1  {64k^2} \right )^{1/2}
        + \frac 2 { 64k^3}\le  \frac 1 {2k}$.
    \end{enumerate}
    Note that, from \ref{lemma3.4.e}, we obtain that
    \begin{equation*}
        ( km) \cdot \tau \Big(\operatorname{supp} (q, \{e_{ss}f_{ii}\}_{1\le s\le k; 1\le i\le m}) \Big)
        \le  (km) \cdot \frac 1 {2k} = \frac m 2
        \le  \ | \{ f_{ii}e_{11} : 2\le i\le m\}|.
    \end{equation*}
    By Lemma \ref{lemma3.4}, we can further assume that $ \operatorname{supp} (q, \{e_{ss}f_{ii}\}_{1\le s\le k; 1\le i\le m}) \le \sum_{i=2}^m f_{ii} e_{11}$, which implies \ref{lemma3.4.i}
    \begin{equation*}
        q = q (\sum_{i=2}^m f_{ii} e_{11})= (\sum_{i=2}^m f_{ii}
        e_{11})q.
    \end{equation*}
    Now \ref{lemma3.4.ii} follows from \ref{lemma3.4.c}, the choice of $y$, and \ref{lemma3.4.d}.
\end{proof}

An easy exercise of spectral theory is needed.

\begin{lemma}\label{lemma7}
    Let $\mathcal N_0\subseteq \mathcal M$ be a subfactor of type ${\rm I}_3$ and $\{e_{ij}\}_{1\le i,j\le 3}$  a system of matrix units of $\mathcal N_0$. Let $  z_1$ and $z_2$ be self-adjoint elements in $\mathcal N_0^{\prime}\cap \mathcal M$ and $y$   an element in $\mathcal N_0^{\prime}\cap \mathcal M$. Assume that
     \begin{equation*}
        a= e_{11}+2e_{22}+3e_{33}
     \end{equation*} and
     \begin{equation*}
     \begin{aligned}
     b=  & \ z_1e_{11} & + & \ ye_{12}   & + &  \ e_{13} &+\\
         & \ y^*e_{21}& + & \ z_2e_{22}  & + & \ e_{23} &+\\
         & \  \ \ \ e_{31}   & + & \ \ \   e_{32}    & + & \ e_{33} .
     \end{aligned}
     \end{equation*}
     Then
     \begin{equation*}
        W^*(y, z_1,z_2 , \mathcal N_0 ) \subseteq W^*(a+ib).
     \end{equation*}
\end{lemma}

\begin{proof}
Apparently, $e_{11}, e_{22}, e_{33}$ are in $W^*(a)$. Observe that
$e_{ii}be_{33}=e_{i3}$ and $e_{33}be_{ii}=e_{3i}$ for $i=1,2,3$. We
have $e_{i3}e_{3j}=e_{ij}$ is in $ W^*(a+ib)$ for $1\le i,j\le 3$.
Thus $\mathcal N_0\subseteq  W^*(a+ib)$. From the fact that
$e_{i1}be_{1i}= z_1e_{ii}$ for $i=1,2,3$, it follows that $z_1=
z_1e_{11}+ z_1e_{22}+ z_1e_{33}$ is in $W^*(a+ib)$. Similarly, it
can be shown that $z_2, y\in W^*(a+ib)$. Therefore,
\begin{equation*}
        W^*(y, z_1,z_2 , \mathcal N_0) \subseteq W^*(a+ib).
     \end{equation*}
     This ends the proof.
\end{proof}

The next lemma, used repeatedly in the proof of Proposition \ref{prop3.9-gamma}, is probably well-known. However, we are unable to find a reference.  For the sake of completeness and to make the paper self-contained, we provide an elementary proof.

\begin{lemma}\label{lemma3.2}
Suppose that $\mathcal M_1$ is a finite von Neumann algebra and $\mathcal M_2\subseteq  \mathcal M_1$ is a von Neumann subalgebra containing the identity  $I$ of $\mathcal M_1$. If $\mathcal M_2$ is diffuse, then $\mathcal M_1$ is diffuse.
\end{lemma}

\begin{proof}
    Assume that $\mathcal M_1$ is not diffuse. Then there exists a nonzero minimal projection $p$ in $\mathcal M_1$. Let $C_p$ be the central support of $p$ in $\mathcal M_1$. By Proposition 6.4.3 of \cite{Kadison2}, $\mathcal M_1 C_p$ is a factor. As $\mathcal M_1$ is finite, $\mathcal M_1 C_p$ is a factor of type ${\rm I}_k$ for some positive integer $k$.

    Define
    $$
        q = \bigvee\{e: e \text{ is a projection in } \mathcal M_2 \text{ such that } eC_p =0 \}.
        $$
    Then $I-q$ is a nonzero projection in $\mathcal M_2$ satisfying, for every nonzero subprojection $f$ of $I-q$ in $\mathcal M_2$, $fC_p \ne 0$. If $\mathcal M_2$ contains no minimal projections, then there exists a family of mutually orthogonal nonzero projections $f_1, \ldots, f_{k+1}$ in $\mathcal M_2$ such that $I-q=f_1+\cdots+ f_{k+1}$. The choice of $q$  ensures that $f_1C_p, \ldots,  f_{k+1}C_p$ is a family of mutually orthogonal nonzero projections in $\mathcal M_1 C_p$, which contradicts the fact that $\mathcal M_1 C_p$ is a factor of type ${\rm I}_k$. This completes the proof.
\end{proof}

Central sequence algebras of type ${\rm II}_1$ factors, developed by McDuff in \cite{McDuff}, play an essential role in the paper. Inspired by a method similar to Theorem 5 of \cite{McDuff} and the subsequent comment, we prove the following result, which is a slight modification of Lemma 3.5 of \cite{FGL} (or Theorem A.6.5 of \cite{SS}) by removing the condition that $\mathcal M$ is separable. Recall that a finite von Neumann algebra with a faithful, normal, tracial state is separable if it has a separable predual, which is equivalent to it being countably generated (see \cite {Dix2} Exercise I.7.3 b and c).

\begin{lemma}\label{may20Lemma2.5}
    Let $\mathcal M$ be a type ${\rm II}_1$ factor with  trace $\tau$. Suppose that $\mathcal Q$ is a separable irreducible ${\rm II}_1$ subfactor of $\mathcal M$. If $\mathcal Q'\cap \mathcal M^\omega \ne \mathbb CI$, then  $\mathcal Q'\cap \mathcal M^\omega  $ is diffuse.
\end{lemma}

\begin{proof}
    Assume that  $\mathcal Q'\cap \mathcal M^\omega \ne \mathbb CI$ and $\mathcal Q'\cap \mathcal M^\omega  $ is not diffuse. Note that $\mathcal M^\omega$ is a type ${\rm II}_1$ factor with a natural trace $\tau_\omega$ and $\mathcal M\subseteq \mathcal M^\omega$. Let $E_{\mathcal M}^{\mathcal M^\omega}: \mathcal M^\omega\rightarrow \mathcal M$ be the trace-preserving conditional expectation from $\mathcal M^\omega$ onto $\mathcal M$. Suppose that $\{a_m\}_{m=1}^\infty$ is a countable family of self-adjoint generators of $\mathcal Q$.

    Let $(q_n)_\omega\ne 0, I$ be a minimal projection in  $\mathcal Q'\cap \mathcal M^\omega  $   and $z=(Y_{n})_\omega$ the central support of $(q_n)_\omega$   in  $\mathcal Q'\cap \mathcal M^\omega  $. Then $ z(\mathcal Q'\cap \mathcal M^\omega)$ is a factor of type ${\rm I}_k$ for some positive integer $k$ by Proposition 6.4.3 of \cite{Kadison2}. Assume that $\tau_\omega((q_n)_\omega)=r$ with $0<r<1$, then $\tau_\omega(z) = kr\le 1$. We can further assume that $q_n$ and $Y_n$ are  projections in $\mathcal M$ with $\tau(q_n)=r$ and $\tau(Y_n)=kr$ for each $n\ge 1$ by Theorem A.5.3 in \cite {SS}.

    We claim that  $\mathcal Q'\cap \mathcal M^\omega  $  is a nonseparable subspace of $ \mathcal M^\omega  $ with respect to $\|\cdot\|_{2,\tau_\omega}$,  the trace norm of $\mathcal M^\omega$.  Assume, to the contrary, that $\mathcal Q'\cap \mathcal M^\omega  $ is separable with respect to $\|\cdot\|_{2,\tau_\omega}$ and assume that  $\{(y_n^{(m)})_\omega\}_{m=1}^\infty$ is a dense subset of $\mathcal Q'\cap \mathcal M^\omega  $.   As $\mathcal Q'\cap\mathcal M=\mathbb CI$, $(q_n)_\omega\notin \mathcal M$, whence $\delta=\|(q_n)_\omega- E_{\mathcal M}^{\mathcal M^\omega}((q_n)_\omega)\|_{2, \tau_\omega} >0$. It follows that
    \begin{equation*}
        \delta =\|(q_n)_\omega- E_{\mathcal M}^{\mathcal M^\omega}((q_n)_\omega)\|_{2, \tau_\omega}\le \|(q_n)_\omega -y\|_{2,\tau_\omega} =\lim_{n\rightarrow \omega} \|q_n-y\|_2, \qquad  \forall \ y\in \mathcal
        M.
    \end{equation*}
    Combining it with the fact that $(q_n)_\omega$ is in $ \mathcal Q'\cap \mathcal M^\omega  $, for each $n\ge 1$  we let $k_n\in\mathbb N$  be  such that
    \begin{equation*}
        \|q_{k_n}- y_n^{(m)} \|_2 \ge \delta /2 \quad \mbox{ and } \quad \|[q_{k_n}, a_m]\|_2 \le 1/n, \qquad \forall \ 1\le m\le
        n.
    \end{equation*}
    Therefore, $(q_{k_n})_\omega \in \mathcal Q'\cap \mathcal M^\omega  $ and $\|(q_{k_n})_\omega- (y_n^{(m)})_\omega\|_{2,\tau_\omega}\ge \delta/2>0$ for $m\ge 1$. This contradicts with the assumption that  $\{(y_n^{(m)})_\omega\}_{m=1}^\infty$ is   dense in $\mathcal Q'\cap \mathcal M^\omega $. Hence  $\mathcal Q'\cap \mathcal M^\omega  $  is  nonseparable.

    Observe that $E_{\mathcal M}^{\mathcal M^\omega}(\mathcal Q'\cap \mathcal M^\omega)\subseteq \mathcal Q'\cap \M =\mathbb C I$. Thus $E_{\mathcal M}^{\mathcal M^\omega}((q_n)_\omega) =\tau_\omega((q_n)_\omega)=r,$ which implies that
    \begin{equation*}
        \begin{aligned}
        \lim_{n\rightarrow \omega}\tau(q_ny) &=\tau_\omega((q_n)_\omega y)= \tau_\omega(E_{\mathcal M}^{\mathcal M^\omega}((q_n)_\omega y))=
        \tau_\omega(E_{\mathcal M}^{\mathcal M^\omega}((q_n)_\omega) y)\\&=\tau_\omega(\tau_\omega((q_n)_\omega) y) = \tau_\omega((q_n)_\omega)\tau(y)=r \ \tau(y), \qquad \text{ for all $y \in \mathcal
        M$.}
    \end{aligned}
    \end{equation*}


    Combining it with that $(q_n)_\omega$ is in $ \mathcal Q'\cap \mathcal M^\omega  $,  for each $n\ge 1$  we let $j_n\in\mathbb N$  be  such that
    \begin{equation*}
        |\tau(q_{j_n} Y_n)-r\cdot kr |\le 2^{-n} \quad \mbox{ and } \quad \max_{1\le m \le n}\|[q_{j_n}, a_m]\|_2 \le 2^{-n}.
    \end{equation*}
    Hence, $( q_{j_n} Y_n)_\omega= ( q_{j_n} )_\omega z$ is a projection in $ z(\mathcal Q'\cap \mathcal M^\omega)$ with trace $kr^2$. On the other hand,  $ z(\mathcal Q'\cap \mathcal M^\omega)$ is   a type I$_k$ factor with a minimal projection $(q_n)_\omega $ of trace $r$. Therefore, $kr^2\ge r$, whence $kr=1$ and $z=I$. This means that $\mathcal Q'\cap \mathcal M^\omega  $ is a type I$_k$ factor, which contradicts the fact that  $\mathcal Q'\cap \mathcal M^\omega  $  is nonseparable. This ends the proof of the lemma.
\end{proof}

The following \Cref{Gamma-2} aims to simplify the proof of \Cref{prop3.9-gamma}. Additionally, the proof of this lemma draws inspiration from that of Proposition 5.10 in \cite{Shen}. The matrix tricks used to obtain the ``almost single generator'' here are also related to the proof of \cite{GeShen}.

\begin{lemma} \label{Gamma-2}
    Let $\mathcal M$ be a type ${\rm II}_1$ factor with trace $\tau$. Suppose that $W^*(x)'\cap \mathcal M^\omega $ is  diffuse for every $x$ in $\mathcal{M}$.

    Then for any finitely many elements $x_1,\ldots, x_n$ in $\mathcal{M}$ and $\varepsilon>0$, there exists an element $z$ in $\mathcal{M}$ such that $\operatorname{dist}_{\Vert\cdot\Vert_2}(x_j,W^*(z)) < \varepsilon$ for every $j=1,\ldots,n$.
\end{lemma}

\begin{proof}
    Let $\mathcal N_0$ be a type ${\rm I}_3$ subfactor of $\mathcal M$ and $\mathcal P =\mathcal N_0 '\cap \mathcal M$ so that $\mathcal M \cong\mathcal P \otimes \mathcal N_0$. 
    Assume that $\{e_{ij}\}_{i,j=1}^3$ is a system of matrix units of $\mathcal N_0$. Then there is a family of elements $\{y_{ij}^{(r)}\}_{1\le i,j\le 3; 1\le r\le n}$ in $\mathcal P $ such that
    \begin{equation} \label{eq-xr-yrij}
        x_r= \sum_{1\le i,j\le 3} y_{ij}^{(r)}e_{ij}, \qquad \forall \ 1\le r\le n.
    \end{equation}
    List elements in $\{ y_{ij}^{(r)} \}_{1\le i,j\le 3; 1\le r\le n} $ as $ \{y_r\}_{1\le r\le 9n}$. To proceed, we prove a claim as follows.
    
    \vspace{0.1in}

    \textbf{Claim \ref{Gamma-2}.1.} \emph{There exist a family of projections $\{q_r\}_{r=1}^{9n}$ in $\mathcal M$ and a family of commuting subfactors $\mathcal N_r$ of type ${\rm I}_{m_r}$ for $r=1, \dots, 9n$ satisfying
    \begin{equation*}
        \mathcal N_r\subseteq (\mathcal N_0\cup \mathcal N_1\cup \cdots \cup \mathcal N_{r-1})'\cap \mathcal M,
    \end{equation*}
    where each $\mathcal N_r$ is equipped with a system of matrix units $\{\{f_{ij}^{(r)}\}_{i,j=1}^{m_r}$, 
    such that, for each $r$,
    \begin{enumerate} [label= {\rm(\arabic*)} ]
        \item  \label{claim3.8.1} $q_r = q_r \left (\sum_{i=2}^{m_r} f_{i,i}^{(r)}\right )f_{1,1}^{(r-1)}\cdots f_{1,1}^{(1)}e_{11} = \left(\sum_{i=2}^{m_r} f_{i,i}^{(r)}\right )f_{1,1}^{(r-1)}\cdots f_{1,1}^{(1)}e_{11} q_r$, and
        \item  \label{claim3.8.2}  $\operatorname{dist}_{\|\cdot\|_2} (y_r, W^*(q_r, \mathcal N_0,\mathcal N_1,\ldots, \mathcal N_r)) < \varepsilon/9$.
    \end{enumerate}}
    
    \vspace{0.1in}

    By induction, we prove Claim \ref{Gamma-2}.1 in two steps.       

    \noindent\textbf{Step One.} Assume that the claim is true for $r=1$. Write
    \begin{equation*}
    a= e_{11}+2e_{22}+3e_{33} \quad \text{ and } \quad b_1=    y_1e_{12}+e_{13}+
        y_1^*e_{21}+e_{23}+e_{31}+e_{32}+e_{33}.
    \end{equation*}
    Then \Cref{lemma7} implies that $W^*(y_1, \mathcal N_0)\subseteq W^*(a,b_1)$. By the assumption and \Cref{lemma3.2},
    \begin{equation*}
        W^*(y_1, \mathcal N_0)'\cap \mathcal M^\omega  \supseteq W^*(a+ib_1)' \cap \mathcal M^\omega\text{ is diffuse.}
    \end{equation*}
    Applying Lemma \ref{lemma6} for $\mathcal A=\mathcal N_0$ and $x=y_1$,  there exist an integer $m_1>0$, a type ${\rm I}_{m_1}$ subfactor $\mathcal N_1$ of $\mathcal N_0'\cap \mathcal M$, a system of matrix units $\{f_{i,j}^{(1)}\}_{i,j=1}^{m_1}$ of $\mathcal N_1$,  and a projection $q_1$ in $\mathcal M$ such that
    \begin{enumerate} [label= {\rm(\arabic*)} ]
      \item   $q_1 = q_1 (\sum_{i=2}^{m_1} f_{ i,i}^{(1)}) e_{11}= (\sum_{i=2}^{m_1} f_{i,i}^{(1)}) e_{11}q_1$;
      \item   $\operatorname{dist}_{\|\cdot\|_2} (y_1, W^*(q_1, \mathcal N_0,\mathcal N_1)) < \varepsilon /9$.
    \end{enumerate}
    This completes the first step of the induction proof.

    \noindent\textbf{Step Two.}  Assume that the claim is true for $r=k$, where $1\le k< 9n$, i.e., the desired $\{m_r\}_{r=1}^{k}$, $\{\mathcal N_r\}_{r=1}^{k}$,  $\{\{f_{i,j}^{(r)}\}_{i,j=1}^{m_r}\}_{r=1}^{k}$ and $\{q_r\}_{r=1}^{k}$ have been obtained. Notice that $\{\mathcal N_r\}_{r=1}^{k}$ is a family of commuting  subfactors in $\mathcal N_0^{\prime}\cap \mathcal M$. So $W^*(\mathcal N_1,\ldots, \mathcal N_k)$ is a subfactor of type {\rm I}$_{m_1\cdots m_k}$, which has two self-adjoint generators   $z_{1},z_{2}$. Moreover, $W^*(\mathcal N_0, \mathcal N_1,\ldots, \mathcal N_k)$ is   a subfactor of type {\rm I}$_{3m_1\cdots m_k}$ with a system of matrix units
    \begin{equation*}
        \{ f_{i_kj_k}^{(k)} f_{i_{k-1}j_{k-1}}^{(k-1)}\cdots
         f_{i_1j_1}^{(1)}e_{ij}  \}_{1\le i,j\le 3; 1\le i_1,j_1\le m_1; \ldots; 1\le i_k, j_k\le
         m_k}.
    \end{equation*}
    Define a self-adjoint operator $b_2$ in the form
    \begin{equation*}
        \begin{aligned}
            b_2= \  & z_{ 1}e_{11} & +  & y_{k+1}e_{12} &+ & e_{13} & +   \\
            & y_{k+1}^*e_{21} &+ & z_2e_{22} & + & e_{23} & +\\
            & e_{31} & + & e_{32} &   + & e_{33}.
    \end{aligned}
    \end{equation*}
    Recall that  $a= e_{11}+2e_{22}+3e_{33} $. Thus \Cref{lemma7} implies that
    \begin{equation} \label{subset-a+ib2}
        W^*(y_{k+1}, \mathcal N_k, \ldots, \mathcal N_1, \mathcal N_0) = W^*(y_{k+1}, z_{1},z_{2}, \mathcal N_0)\subseteq W^*(a+ib_2).
    \end{equation}
    By the assumption, we have that $W^*(a+ib_2)' \cap \mathcal M^\omega$ is diffuse. Combining this with \eqref{subset-a+ib2} and Lemma \ref{lemma3.2}, we obtain that
    \begin{equation*}
        W^*(y_{k+1}, \mathcal N_k, \ldots, \mathcal N_1, \mathcal N_0)'\cap \mathcal M^\omega    \text{ is diffuse.}
    \end{equation*}
    By  writing $\mathcal A=W^*(\mathcal N_0, \mathcal N_1,\ldots, \mathcal N_k)$ and $x=y_{k+1}$ in  Lemma \ref{lemma6},  there exist a positive integer $m_{k+1}$, a type {\rm I}$_{m_{k+1}}$ subfactor $\mathcal N_{k+1}$ of $W^*(\mathcal N_0,\mathcal N_1,\ldots,\mathcal N_k)'\cap \mathcal M$, a system of matrix units $\{f_{ij}^{(k+1)}\}_{i,j=1}^{m_{k+1}}$ of $\mathcal N_{k+1}$, and a projection $q_{k+1}$ in $\mathcal M$ such that
    \begin{enumerate}[label={\rm (\arabic*)}]
        \item   $q_{k+1} = q_{k+1} \left (\sum_{i=2}^{m_{k+1}} f_{i i}^{(k+1)}\right )f_{11}^{(k)}\cdots f_{1 1}^{(1)}e_{11} = \left (\sum_{i=2}^{m_{k+1}} f_{ i i}^{(k+1)} \right ) f_{11}^{(k)}\cdots f_{1 1}^{(1)}e_{11}q_{k+1}$;
        \item   $\operatorname{dist}_{\|\cdot\|_2} (y_{k+1}, W^*(q_{k+1}, \mathcal N_0,\mathcal N_1,\ldots, \mathcal N_{k+1})) < \varepsilon /9$.   
    \end{enumerate} 
    Thus, the claim is true for $r = k+1$, which completes the induction step and hence the proof of Claim \ref{Gamma-2}.1.

    (\emph{End of the proof of $\Cref{Gamma-2}$}.) By Claim \ref{Gamma-2}.1, we obtain $\{\mathcal N_r\}_{r=1}^{9n}$, $\{\{f_{ij}^{(r)} \}_{i,j=1}^{m_r}\} _{r=1}^{9n}$, and $\{q_r\}_{r=1}^{9n}$ with the properties as listed in Claim \ref{Gamma-2}.1.  Notice that $W^*(\mathcal N_1,\ldots, \mathcal N_{9n})$ is a subfactor of type ${\rm I}_{m_1\cdots m_{9n}}$. Assume that $\tilde z_{1},\tilde z_{2}$ are two self-adjoint generators of $W^*(\mathcal N_1,\mathcal N_2, \ldots, \mathcal N_{9n})$. The Conclusion \ref{claim3.8.1} of Claim \ref{Gamma-2}.1 entails that $q_1,\ldots, q_{9n}$ are mutually orthogonal  sub-projections of $e_{11}$. The spectral theorem for self-adjoint operators implies that
    \begin{equation*}
        \{  q_1, \ldots, q_{9n}\}\subseteq W^*(  q_1 + q_2/2 + \cdots+ q_{9n}/ 2^{9n }  ).
    \end{equation*}
    Recall that  $a=  e_{11} +2e_{22}+3e_{33}$.  Define a self-adjoint operator $b_3$ in the form
    \begin{equation*}
        \begin{aligned}
            b_3=    (q_1 + q_2/2 + \cdots+ q_{9n}/ 2^{9n } ) & e_{11} & +   \tilde z_{ 1}e_{12} & +   e_{13}+   \\
            \tilde z_{ 1}& e_{21} & + \tilde z_2e_{22} & + e_{23}+\\
            & e_{31}   & + \ \ e_{32}     & +e_{33}.
        \end{aligned}
    \end{equation*}
    A similar proof to \Cref{lemma7} yields that
    \begin{equation}  \label{subset-a+ib3}
        W^*(q_1,\ldots, q_{9n}, \mathcal N_0, \mathcal N_1,\ldots, \mathcal N_{9n})\subseteq W^*(a+ib_3).
    \end{equation}
    By Conclusion \ref{claim3.8.2} of Claim  \ref{Gamma-2}.1, there is an element $w_k$ in $ W^*(q_1,\ldots, q_{9n}, \mathcal N_0, \mathcal N_1,\ldots, \mathcal N_{9n})$ such that $\|y_k-w_k\|_2 < \varepsilon /9$ for each $1\le k\le 9n$. 
    Since $\{ y_{ij}^{(r)} \}_{1\le i,j\le 3; 1\le r\le n}$ was listed as $ \{y_k\}_{1\le k \le 9n}$, we can  rename $ \{w_k\}_{1\le k \le 9n}$ as $\{w_{ij}^{(r)}\}_{1\le i,j\le 3; 1\le r\le n}$ correspondingly with
    \begin{equation*}
        \|y_{ij}^{(r)}- w_{ij}^{(r)}\|_2 < \varepsilon /9.
    \end{equation*}
    Write $z = a + ib_3$ and $w_r = \sum^3_{i,j=1} w_{ij}^{(r)}e_{ij}$ for each $1\le r\le n$. It follows that
    \begin{equation*}
        \begin{aligned}
        \|x_r - w_r\|_2 & 
        \le \sum_{i,j=1}^3 \left\|(y_{ij}^{(r)}-w_{ij}^{(r)})e_{ij} \right\|_2  
        \le \sum_{i,j=1}^3 \left\| y_{ij}^{(r)}-w_{ij}^{(r)} \right\|_2 <\varepsilon.
        \end{aligned}
    \end{equation*}
    Since each $w_r$ lies in $W^*(z)$, we obtain that $\operatorname{dist}_{\Vert\cdot\Vert_2}(x_r,W^*(z)) < \varepsilon$ for every $r=1,\ldots,n$. This completes the proof. 
\end{proof}

If $\mathcal M$ is a singly generated type ${\rm II}_1$ factor, then, by \cite{Dixm}, that $\mathcal M$ has property $\Gamma$ is equivalent to $W^*(x)'\cap \mathcal M^\omega\ne \mathbb CI$ for all $x\in\mathcal M$. In fact  a more general statement, without the assumption on cardinality of generators of $\mathcal M$, is still valid. We develop this in the following proposition. 

\begin{proposition}\label{prop3.9-gamma}
    Let $\mathcal M$ be a type ${\rm II}_1$ factor with trace $\tau$. Then the following statements are equivalent:
    \begin{enumerate}[label={\rm (\roman*)}]
        \item \label{prop3.8.i}  $\mathcal M$ has property $\Gamma$.
        \item \label{prop3.8.ii}  For every $x$ in $\mathcal M$, $W^*(x)'\cap \mathcal M^\omega $ is  diffuse.
        \item \label{prop3.8.iii}  For every $x$ in $\mathcal M$ and every nonzero projection $p$ in $\mathcal M$, 
        $$W^*(p x p)'\cap (p\mathcal M p)^\omega \ne \mathbb C p. $$
        \item \label{prop3.8.iv}  For every $x $ in $\mathcal M$,  $W^*(x )'\cap \mathcal M^\omega \ne \mathbb C I$.
    \end{enumerate}
\end{proposition}

\begin{proof}  \ref{prop3.8.i} $\Rightarrow$ \ref{prop3.8.ii}.
    Assume that \ref{prop3.8.i} holds.  Let $x$ be an element in $\mathcal M$. By Proposition 7.1 in \cite{CPSS}, there exists a separable subfactor $\mathcal M_1$ with property $\Gamma$ such that $ x\in \mathcal M_1\subseteq \mathcal M$. It follows that $\mathcal M_1'\cap \mathcal M_1^\omega$ is diffuse (see \cite{Dixm} or Lemma \ref{may20Lemma2.5}). Hence, by Lemma \ref{lemma3.2},  $W^*(x )'\cap \mathcal M^\omega \supseteq  \mathcal M_1'\cap \mathcal M_1^\omega$ is diffuse, i.e., \ref{prop3.8.ii} is true.

    \ref{prop3.8.ii} $\Rightarrow$ \ref{prop3.8.i}.
    Assume that \ref{prop3.8.ii} holds. Suppose that $x_1,\ldots, x_n$ are elements in $\mathcal M$ and $\varepsilon>0$ is a positive number. Thus, we can directly follow the proof of \Cref{Gamma-2} and its notation. 
    By Conclusion \ref{claim3.8.2} of Claim \ref{Gamma-2}.1, there is an element $w_k$ in $ W^*(q_1,\ldots, q_{9n}, \mathcal N_0, \mathcal N_1,\ldots, \mathcal N_{9n})$ such that $\|y_k-w_k\|_2 < \varepsilon$ for each $1\le k\le 9n$. From \eqref{subset-a+ib3} and Lemma \ref{lemma3.2}, we obtain that  $ W^*(q_1,\ldots, q_{9n}, \mathcal N_0, \mathcal N_1,\ldots, \mathcal N_{9n})'\cap \mathcal M^\omega$ is diffuse. By the inclusion 
    \begin{equation*}
        W^*(q_1,\ldots, q_{9n}, \mathcal N_0, \mathcal N_1,\ldots, \mathcal N_{9n})'\cap \mathcal M^\omega \subseteq (\mathcal N_0'\cap \mathcal M)^\omega
    \end{equation*}
    there exists a unitary element $u$ in $\mathcal N_0'\cap \mathcal M$ such that $\tau(u)=0$ and $\|[w_k, u]\|_2 < \varepsilon $ for $1\le k\le 9n$. Rename $ \{w_r\}_{1\le r\le 9n}$ as $\{w_{ij}^{(r)}\}_{1\le i,j\le 3; 1\le r\le n}$ correspondingly such that
    \begin{equation*}
        \|y_{ij}^{(r)}- w_{ij}^{(r)}\|_2 < \varepsilon.
    \end{equation*}
    Therefore, for each $1\le r\le n$, it follows from \eqref{eq-xr-yrij} that
    \begin{equation*}
        \begin{aligned}
        \left\|[x_r, u]\right\|_2
        & \le \sum_{i,j=1}^3 \left\|[y_{ij}^{(r)}e_{ij}, u]\right\|_2 
        \le \sum_{i,j=1}^3 \left\|[y_{ij}^{(r)}, u]  \right\|_2 \\
        & \le \sum_{i,j=1}^3 \left\|[w_{ij}^{(r)}, u]  \right\|_2 + \sum_{i,j=1}^3 \left\|[y_{ij}^{(r)}-w_{ij}^{(r)}, u]  \right\|_2   
        < 27\varepsilon.
        \end{aligned}
    \end{equation*}
    By the definition of property $\Gamma$, \ref{prop3.8.i} of Proposition \ref{prop3.9-gamma} is proved.

    \ref{prop3.8.i} $\Rightarrow$ \ref{prop3.8.iii}.
    Assume that $\mathcal M$ has property $\Gamma$. Let $x$ be an element in $\mathcal M$ and $p$ a nonzero projection in $\mathcal M$. By virtue of Lemma \ref{lemma2.2}, $p\mathcal Mp$ has property $\Gamma$. This entails the inequality $W^*(p x p)'\cap (p\mathcal M p)^\omega \ne \mathbb C p$.

    \ref{prop3.8.iii} $\Rightarrow$  \ref{prop3.8.ii}.
    Let $x $ be an element in $\mathcal M$. Define $\mathcal N=W^*(x)'\cap \mathcal M$. If $\mathcal N$ contains no minimal projections, then $\mathcal N$ is diffuse, whence $W^*(x )'\cap \mathcal M^\omega \supseteq \mathcal N$ is diffuse by Lemma  \ref{lemma3.2}. Otherwise,   we assume that $\{p_{n}\}_{n=1}^N$ is a maximal family of mutually orthogonal minimal projections in $\mathcal N$, where $1\le N\le \infty$. Define $p_0=I-\sum_{n=1}^N p_n$. Then $p_0\mathcal Np_0$ is diffuse.   For each $n\ge 1$,  it is not hard to verify that $W^*(p_n x p_n )$ is an irreducible subfactor of $p_n \mathcal M p_n$. The assumption $W^*(p_n x p_n )'\cap (p_n \mathcal M p_n )^\omega\ne \mathbb Cp_n $ and Lemma \ref{may20Lemma2.5} guarantee that $W^*(p_n x p_n )'\cap (p_n \mathcal M p_n )^\omega$ is diffuse. Obviously,
    \begin{equation*}
        W^*(x )'\cap \mathcal M^\omega \supseteq p_0\mathcal Np_0 \oplus \Big( \oplus_{n=1}^N \big(W^*(p_nx p_n)'\cap (p_n\mathcal Mp_n)^\omega
        \big)\Big)\supseteq \mathbb CI.
    \end{equation*}
    This entails that $W^*(x )'\cap \mathcal M^\omega$ is diffuse, by virtue of Lemma \ref{lemma3.2}.

    \ref{prop3.8.iii} $\Rightarrow$ \ref{prop3.8.iv}.   This is trivial.

    \ref{prop3.8.iv} $\Rightarrow$ \ref{prop3.8.iii}.
    We use a contrapositive proof here. Assume that   \ref{prop3.8.iii}  is false. Thus there exist  two self-adjoint elements $y_1$ and $ y_2$ in $\mathcal M$ and a nonzero projection $p$ in $\mathcal M$ such that
    \begin{equation*}
        W^*(p(y_1+iy_2)p)'\cap (p\mathcal M p)^\omega = \mathbb C p.
            \end{equation*}
    Denote $p$ by $p_0$. As $\mathcal M$ is a type ${\rm II}_1$ factor, there exists a family of mutually orthogonal subprojections $p_1, \ldots, p_{k-1}, p_k$ of $I-p_0$ such that
    \begin{equation*}
        \tau(p_0)=\tau(p_1)=\cdots =\tau(p_{k-1})\ge \tau(p_k) \quad \text{ and } \quad p_0+p_1+\cdots +p_k=I.
    \end{equation*}
    Furthermore, there exists a family of partial isometries $v_1,\ldots, v_k$ in $\mathcal M$ such that
    \begin{align}\text{ $v_iv^*_i=p_0$, $v_i^*v_i=p_i$ for $1\le i\le k-1$ and
     $v_kv_k^*\le p_0$, $v_k^*v_k=p_k$.} \label{may21Equ2}\end{align}

    Without loss of generality, we assume that $\|p_0y_1p_0\|<1$. Define two self-adjoint elements $x_1$ and $x_2$ in $\mathcal M$ as follows:
    \begin{equation*}
        \begin{aligned}
            x_1 & =  (p_0+p_0y_1p_0)+ 2p_1+3p_2+ \cdots+ (k+1)p_k \quad \text{ and} \\
            x_2 & =  p_0y_2p_0 + (v_1 + v_2+ \cdots + v_k) + ( v_1 + v_2 + \cdots + v_k)^*.
        \end{aligned}
    \end{equation*}
    By spectral theory, we obtain that $p_0, p_1, \ldots, p_k, p_0y_1p_0$ are in $W^*(x_1)$ and $p_0y_2p_0, v_1,\ldots, v_k$ are in $W^*(x_1+i x_2)$.

    We claim that $W^*(x_1+i x_2)'\cap \mathcal M^\omega=\mathbb CI$. In fact, assume that $(q_n)_\omega$ is a projection in $W^*(x_1, x_2)'\cap \mathcal M^\omega$. From the fact that $p_0,  p_0y_1p_0$ are in $W^*(x_1)$, we conclude that
    \begin{align} 
        \text{$p_0 $ commutes with $(q_n)_\omega$ in $\mathcal M^\omega$,} \label{may21Equ1}
    \end{align} 
    whence $p_0(q_n)_\omega p_0$ is a projection in $\mathcal M^\omega$.   From the fact that  $p_0y_1p_0, p_0y_2p_0$ are in $W^*(x_1,x_2)$, it follows that $p_0y_1p_0$ and $p_0y_2p_0$ commute with $(q_n)_\omega$  in $\mathcal M^\omega$. Therefore
    \begin{equation*}
        p_0(q_{n})_\omega p_0\in W^*(p_0y_1p_0,p_0y_2p_0)'\cap (p_0\mathcal Mp_0)^\omega=\mathbb C
        p_0.
    \end{equation*}
    Thus $p_0(q_{n})_\omega p_0 =0$ or $p_0$. We proceed the proof by considering the following two cases.

    {\em Case $1$}. Assume that $p_0(q_{n})_\omega p_0 =0$. Since $v_1,\ldots, v_k$ are in $ W^*(x_1, x_2)$, it follows that $(q_n)_\omega v_i  =v_i (q_n)_\omega$. This, together with (\ref{may21Equ2})  and (\ref{may21Equ1}), implies that $ 0= p_0(q_n)_\omega p_0 v_i  = (q_n)_\omega v_i= v_i (q_{n })_\omega  $, whence $p_i(q_{n })_\omega = v_i^*v_i(q_{n })_\omega=0$ for all $1\le i\le k$. So $(q_{n })_\omega =\sum^{k}_{i=0}p_i(q_{n })_\omega=0$.

    {\em Case $2$}. Assume that $p_0(q_{n})_\omega p_0 =p_0$.  As $v_1,\ldots, v_k\in W^*(x_1, x_2)$, we have the equality $(q_n)_\omega v_i =v_i(q_n)_\omega $.  This, together with  (\ref{may21Equ2})  and (\ref{may21Equ1}), implies that $ v_i  = p_0v_i=p_0(q_{n})_\omega p_0 v_i  =  (q_{n})_\omega   v_i =  v_i (q_{n})_\omega  $, whence $p_i=v_i^*v_i=v_i^* v_i (q_{n})_\omega =p_i(q_{n})_\omega$ for all  $1\le i\le k$. It follows that $(q_{n })_\omega =\sum^{k}_{i=0}p_i(q_{n })_\omega  =\sum^{k}_{i=0}p_i=I$.

    In summary, we conclude that $(q_n)_\omega$ is either $0$ or $I$. Thus $W^*(x_1+i x_2)'\cap \mathcal M^\omega=\mathbb CI$, whence \ref{prop3.8.iv} is false.  This ends the proof of the implication \ref{prop3.8.iv} $\Rightarrow$ \ref{prop3.8.iii}.
  \end{proof}

Now Theorem \ref{thm3.1} follows directly from Proposition \ref{prop3.9-gamma}.

\begin{proof}[Proof of Theorem \ref{thm3.1}]
    The implication ``$\Rightarrow$'' is obvious.
    For the implication ``$\Leftarrow$'',    the assumption implies that $W^*(x )^{\prime}\cap \mathcal M^\omega\ne \mathbb CI$, for every $x $ in $\mathcal M$. Now   Proposition \ref{prop3.9-gamma} guarantees that $\mathcal M$ has property $\Gamma$.
\end{proof}

The following result, implied in  Theorem 2.1 in \cite{Connes}, is well known and its proof is sketched.

\begin{proposition}\label{prop3.9}
    Let $\mathcal M$ be a type ${\rm II}_1$ factor with    trace $\tau$. Let $x $ be an element in $\mathcal M$. Consider the following statements:
    \begin{enumerate}[label={\rm (\roman*)} ]
        \item \label{prop3.10.i}  For any given $\varepsilon>0$, for every nonzero projection $p\in\mathcal M$, there exists a nonzero sub-projection $q$ of $ p$ in $\mathcal M$ satisfying
        \begin{equation*}
            \tau(q) \leq \varepsilon \quad \mbox{ and } \quad \Vert [px p,q] \Vert_{2} \leq \varepsilon \cdot \Vert q\Vert_{2} ,
        \end{equation*}
        \item \label{prop3.10.ii}   $W^*(x )^{\prime}\cap \mathcal M^\omega\ne \mathbb CI$.
    \end{enumerate}
    Then the implication  \ref{prop3.10.i} $ \Rightarrow$ \ref{prop3.10.ii}  holds.
\end{proposition}

\begin{proof}
    Assume that \ref{prop3.10.i} holds. To prove \ref{prop3.10.ii}, it suffices to show that, for any $\varepsilon>0$, there exists a projection $q$ in $\mathcal M$ such that $\tau(q)=1/2$ and  $\|[x , q]\|_2\le \varepsilon$.

    Denote by $\P(\M)$ the set of all the projections in $\mathcal{M}$. Define $\mathcal S$ to be a subset of $\mathcal M$ in the following form:
    \begin{equation*}
        \SS =\{ e\in \P (\M) : 0< \tau(e)\leq \frac{1}{2} \text{ and } \Vert [x ,e] \Vert_{2} \leq \varepsilon \cdot \Vert e\Vert_{2}
        \}.
    \end{equation*}

    For projections $q_1$ and $q_2$ in $\mathcal S$, if $q_1$ is a sub-projection of $q_2$, then define $q_1\preceq q_2$. Thus, the binary relation ``$\preceq$'' is a partial order on $\mathcal S$.

    The Assumption \ref{prop3.10.i} implies that $\mathcal S$ contains a nonzero projection in $\mathcal M$. Moreover, since $\tau$ is normal, each totally ordered chain in $\mathcal S$ has an upper bound in $\mathcal S$. Thus, by Zorn's lemma, there exists a maximal element $q$ in $\mathcal S$.

    If $\tau(q)=1/2$, then the proof is completed. Assume, on the contrary, that $\tau(q)<1/2$. Then there exists a positive number $\delta>0$ with $\tau(q)+\delta<1/2$ and with $0<\delta<\varepsilon/2$. By applying Assumption (i) to $\delta>0$ and $I-q$, there exists a nonzero sub-projection $q_0$ of $I-q$ such that
    \begin{equation*}
        \tau(q_0)\le \delta \quad \mbox{ and } \quad \Vert [(I-q)x (I-q),q_0] \Vert_2 \leq \delta \Vert q_0
        \Vert_2.
    \end{equation*}
    Note that $\tau(q)<\tau(q+q_0)<1/2$ and
    \begin{eqnarray*}\label{q+q_0}
        \Vert [x ,q+q_0] \Vert^2_2&=&\Vert qx (I-q-q_0) \Vert^2_2+\Vert q_0x (I-q-q_0) \Vert^2_2  \\
        &&+\Vert (I-q-q_0)x  q \Vert^2_2+\Vert (I-q-q_0) x  q_0\Vert^2_2   \\
        &\leq & \Vert [x ,q] \Vert^2_2+\Vert [(I-q)x (I-q),q_0] \Vert^2_2   \\
        &\leq & \varepsilon^2\Vert q \Vert^{2}_{2}+\delta^2\Vert q_0 \Vert^{2}_{2}   
        \leq  \varepsilon^2\Vert q+q_0 \Vert^{2}_{2}.
    \end{eqnarray*}
    This implies that $q+q_0\in\mathcal S$. But $q+q_0\in\mathcal S$ contradicts the fact that $q$ is a maximal element in $\mathcal S$. This ends the proof of the proposition.
\end{proof}

A type ${\rm II}_1$ factor $\mathcal M$ with separable predual is called a {\em McDuff factor} if $\mathcal M\simeq \mathcal M \otimes \mathcal R$, where $\mathcal R$ is the hyperfinite II$_1$ factor with separable predual.
Here we provide an answer to Sherman's question in Problem 2.11 of \cite{Sherman}.

\begin{corollary}\label{Sherman}
    Let $n \ge 2$ be a positive integer and $\mathcal M$ a type ${\rm II}_1$ factor with separable predual. Then the following statements are equivalent:
    \begin{enumerate}[label={\rm (\roman*) }]
        \item  \label{cor3.11.i} $\mathcal M$ is a McDuff factor.
        \item  \label{cor3.11.ii}  For any $x\in \mathcal M$, $W^*(x)'\cap \mathcal M^\omega$ is a type ${\rm II}_1$ von Neumann algebra.
        \item  \label{cor3.11.iii}  For any $x\in\mathcal M$, $W^*(x)'\cap \mathcal M^\omega$  unitally contains $\mathcal M_n(\mathbb C)$.
        \item  \label{cor3.11.iv}  For any $x\in\mathcal M$, $W^*(x)'\cap \mathcal M^\omega$  is not abelian.
    \end{enumerate}
\end{corollary}

\begin{proof}
    \ref{cor3.11.i} $\Rightarrow$ \ref{cor3.11.ii} $\Rightarrow$ \ref{cor3.11.iii} $\Rightarrow$ \ref{cor3.11.iv}  is obvious.

     \ref{cor3.11.iv} $\Rightarrow$  \ref{cor3.11.i}.    The assumption (iv) implies that $W^*(x )^{\prime}\cap \mathcal M^\omega\ne \mathbb CI$, for every $x $ in $\mathcal M$. By   Proposition \ref{prop3.9-gamma}, $\mathcal M$ has property $\Gamma$. As $\mathcal M$ has a separable predual,  there exists an element $y$  in $\mathcal M$  such that $W^*(y)=\mathcal M$ by Theorem 6.2 of \cite{PopaGe}. Again from the assumption  \ref{cor3.11.iv}, $W^*(y)'\cap \mathcal M^\omega=\mathcal M'\cap \mathcal M^\omega$ is noncommutative. Now  \ref{cor3.11.i} follows from Theorem 3 of \cite{McDuff}.
\end{proof}

\section{An operator with  spectral gap in a  non-\texorpdfstring{$\Gamma$} \ \  type \texorpdfstring{${\rm II}_1$} \ \  factor}

In this section, we will show the existence of a single operator with spectral gap in a non-$\Gamma$ type  ${\rm II}_1$ factor. The main result, Theorem \ref{prop5.3},  will be proved by a series of lemmas. Let $\mathcal M$ be a type ${\rm II}_1$ factor with  trace $\tau$.

%

\begin{definition}\label{may14Def5.3}
    Let $\{q_k\}^{\infty}_{k=1}$ be a sequence of nonzero projections in $\mathcal M$. Define
    \begin{equation*}
        \mathcal A (\{q_k\}^{\infty}_{k=1}) =\{ x\in \mathcal M : \lim_{k\rightarrow \infty} \frac {\|[x,q_k]\|_2}{\|q_k\|_2} =0\}.
    \end{equation*}
\end{definition}

\begin{lemma}\label{may18Lemma4.2}
    If  $\{q_k\}^{\infty}_{k=1}$ is  a sequence of nonzero projections in $\mathcal M$, then $\mathcal A (\{q_k\}^{\infty}_{k=1})$ is a unital $C^*$-subalgebra of $\mathcal M$.
\end{lemma}

\begin{proof}
    Apparently  $\mathcal A (\{q_k\}^{\infty}_{k=1})$ contains $0$ and $I$. Let $x$ and $y$ be in  $\mathcal A (\{q_k\}^{\infty}_{k=1})$ and $\alpha\in\mathbb C$. Observe that
    \begin{equation*}
        \begin{aligned}
        \|[x^*, q_k]\|_2 & =   \|[x, q_k]\|_2, \\
        \|[\alpha x+y, q_k]\|_2 &\le  |\alpha|\|[x, q_k]\|_2+  \|[y, q_k]\|_2, \\
        \|[xy,q_k]\|_2 &= \|xyq_k- q_kxy\|_2= \|x(yq_k-q_ky)+(xq_k-q_kx)y\|_2 \\
         &\le \|x(yq_k-q_ky)\|_2 + \|(xq_k-q_kx)y\|_2 \\
         &\le \|x\|\|[y, q_k]\|_2 + \|[x,q_k]\|_2\|y\|.
        \end{aligned}
    \end{equation*}
    It follows that   $\mathcal A (\{q_k\}^{\infty}_{k=1})$ is a $*$-algebra.

    Assume that $z\in\mathcal M$ is in the operator norm closure of  $\mathcal A (\{q_k\}^{\infty}_{k=1})$. For any $\delta>0$,  there exists an element $\tilde z$ in $\mathcal A (\{q_k\}^{\infty}_{k=1})$  such that $\|z-\tilde z\|\le \delta$. Hence
    \begin{equation*}
        \limsup_{k\rightarrow \infty} \frac {\|[z,q_k]\|_2}{\|q_k\|_2} \le  \limsup_{k\rightarrow \infty} \frac {\|[\tilde z,q_k]\|_2 + \|[(z-\tilde z), q_k]\|_2}{\|q_k\|_2} \le 0 + 2
        \delta.
    \end{equation*}
    It follows that $z\in \mathcal A (\{q_k\}^{\infty}_{k=1})$. Therefore, $\mathcal A (\{q_k\}^{\infty}_{k=1})$ is a unital $C^*$-subalgebra of $\mathcal M$.
\end{proof}

\begin{lemma}\label{may10Lem5.4}
    Let $\{q_k\}^{\infty}_{k=1}$ be  a sequence of nonzero projections in $\mathcal M$. If $p$ is a projection in $\mathcal A (\{q_k\}^{\infty}_{k=1})$, then
    \begin{equation*}
        \lim_{k\rightarrow \infty} \frac {\|[pxp, pq_kp]\|_2}{\|q_k\|_2}=0, \ \ \forall \ x\in \mathcal A
        (\{q_k\}^{\infty}_{k=1}).
    \end{equation*}
\end{lemma}

\begin{proof}
    For every $x$ in $\mathcal A (\{q_k\}^{\infty}_{k=1})$,
    \begin{equation*}
        \begin{aligned}
        \|[pxp,pq_kp]\|_2 
        & =\|pxp(q_kp-pq_k) + (pxpq_k- q_kpxp)+ (q_kp-pq_k)pxp\|_2\\
        &\le \|pxp\|\|[p, q_k]\|_2+ \|[pxp, q_k]\|_2 + \|[q_k, p]\|_2\|pxp\|.
        \end{aligned}
    \end{equation*}
    Since $p$ and $pxp$ are in $\mathcal A (\{q_k\}^{\infty}_{k=1})$,
    \begin{equation*}
        \lim_{k\rightarrow \infty} \frac {\|[pxp,pq_kp]\|_2
        }{\|q_k\|_2}=0.
    \end{equation*}
    This completes the proof.
\end{proof}

The following Lemma \ref{may14Lemma5.6} is prepared for Lemma \ref{may14Lemma5.7}.

\begin{lemma}\label{may14Lemma5.6}
    If $0\le y\le 1$ is an element in $\mathcal M$, then there exists a projection $e\in W^*(y)$ satisfying $\|y-e\|_2\le 2 \|y-y^2\|_2$. Moreover, $e$ can be chosen to be a subprojection of the range projection of $y$.
\end{lemma}

\begin{proof}
    According to the spectral theorem, the faithful normal tracial state $\tau$ induces a probability measure $\mu$ on $[0,1]$ such that
    \begin{equation*}
        \tau(f(y)) =\int_{[0,1]} f(t) \mathrm{~d} \mu(t)
    \end{equation*}
    for every bounded Borel function $f$ on $[0,1]$.
    Then
    \begin{equation*}
        \begin{aligned}
        \|y-y^2\|_2^2=\tau((y-y^2)^2)  & = \int_{[0,1]} |t-t^2|^2 \mathrm{~d} \mu \ge \frac 1 4  \int_{[0,1/2]} t^2 \mathrm{~d} \mu + \frac 1 4 \int_{(1/2,1]} (1-t)^2 \mathrm{~d} \mu. \\
        \end{aligned}
    \end{equation*}
    Let $e$ be the spectral projection of $y$ corresponding to the interval $(1/2, 1]$. It follows that
    \begin{equation*}
        \begin{aligned}
            \|y-e\|_2^2=\tau((y-e)^2) & =   \int_{[0,1]} |t-\chi_{(1/2,1]}(t)|^2 \mathrm{~d} \mu \\
            & =  \int_{[0,1/2]} t^2 \mathrm{~d} \mu +   \int_{(1/2,1]} (1-t)^2 \mathrm{~d} \mu \\
            & \le  4 \|y-y^2\|_2^2.
        \end{aligned}
    \end{equation*}
    Thus, $e$ is a projection in   $W^*(y)$ satisfying $\|y-e\|_2\le 2 \|y-y^2\|_2$.
\end{proof}

\begin{lemma}\label{may14Lemma5.7}
    Let $\{q_k\}^{\infty}_{k=1}$ be  a sequence of nonzero projections in $\mathcal M$. If $p$ is a projection in $\mathcal A (\{q_k\}^{\infty}_{k=1})$, then there exists a subprojection $e_k$ of $p$ for each $k\ge 1$ such that
    \begin{equation*}
        \lim_{k\rightarrow \infty} \frac
        {\|e_k-pq_kp\|_2}{\|q_k\|_2}=0.
    \end{equation*}
\end{lemma}

\begin{proof}
    Note that
    \begin{equation*}
        \|pq_kp-pq_kpq_kp\|_2= \|  p [p, q_k]q_kp\|_2\le \|[p, q_k]\|_2.
    \end{equation*}
    Applying Lemma \ref{may14Lemma5.6} to $y=pq_kp$ in $p\mathcal{M}p$, we obtain a subprojection $e_k$ of $p$ such that
    \begin{equation*}
        \|pq_kp-e_k\|_2\le 2 \|pq_kp-(pq_kp)^2\|_2,
    \end{equation*}
    whence
    \begin{equation*}
        \lim_{k\rightarrow \infty}
        \frac {\|e_k-pq_kp\|_2}{\|q_k\|_2}\le\lim_{k\rightarrow \infty} \frac {2 \|pq_kp-(pq_kp)^2\|_2}{\|q_k\|_2} \le  \lim_{k\rightarrow \infty}
        \frac {2 \|[p, q_k]\|_2}{\|q_k\|_2} =0.
    \end{equation*}
    This completes the proof.
\end{proof}

The following Lemma \ref{may14Lemma5.75} is prepared for Lemma \ref{may10Lemma5.10}.

\begin{lemma}\label{may14Lemma5.75}
    Let $\mathcal{N}$, with $I_{\mathcal M}\in \mathcal N$,  be a type ${\rm I}_m$ subfactor of $\mathcal M$  and $\{p_{ij}\}_{i,j=1}^m$  be a system of matrix units of $\mathcal N$. For any $x\in \mathcal N'\cap \mathcal M$,
    \begin{equation*}
        \tau(xp_{11}) =\frac 1 m \tau(x).
    \end{equation*}
\end{lemma}

\begin{proof}
    Notice that, for every $1\le i\le m$,
    \begin{equation*}
        \tau(p_{ii}x)= \tau(p_{i1}p_{1i}x) =\tau (p_{1i}x p_{i1})= \tau(p_{1i}p_{i1}x)
        =\tau(p_{11}x).
    \end{equation*}
    Hence
    \begin{equation*}
        \tau(x)= \sum_i \tau(p_{ii} x) = m\tau(p_{11}x).
    \end{equation*}
    This ends the proof.
\end{proof}


\begin{lemma}\label{may10Lemma5.10}
    Let $\{q_k\}^{\infty}_{k=1}\subseteq \M$ be  a sequence of nonzero projections. Suppose that $\mathcal N$, with $I_{\mathcal M}\in\mathcal N\subseteq \mathcal A (\{q_k\}^{\infty}_{k=1})$,  is a subfactor of type ${\rm I}_m$ for some positive integer $m$.  Let $\{p_{ij}\}_{ij=1}^m$ be a system of matrix units of $\mathcal N$. If $p$ is a projection in $\mathcal A (\{q_k\}^{\infty}_{k=1})$  with $p\ge p_{11}$, then
    \begin{equation*}
        \liminf_{k\rightarrow \infty} \frac {\|pq_kp\|_2^2}{\|q_k\|_2^2} \ge \frac 1
        m.
    \end{equation*}
\end{lemma}

\begin{proof}
    Let $\mathcal U$ be a finite group consisting of unitary elements in $\mathcal N$ such that $\mathcal N$ is a linear span of $\mathcal U$ (see Lemma 2.4.1 in \cite{SS2}). Define
    \begin{equation*}
        \tilde q_k = \frac 1 {|\mathcal U|} \sum _{u\in\mathcal U} uq_ku^*, \qquad \text{ for each } k\ge
        1,
    \end{equation*}
    where $|\mathcal U|$ is the cardinality of $\mathcal U$. It is easy to verify that
    $\tilde q_k$ is a positive element  in $\mathcal N'\cap \mathcal M$  for  $k\ge 1$ such that
    \begin{equation}\label{may15Equ4.1}
    \lim_{k\rightarrow \infty} \frac {\|q_k-\tilde q_k\|_2}{\|q_k\|_2} =0,
    \end{equation}
    whence
    \begin{equation}\label{may15Equ4.2}
    \|\tilde q_k\|_2 =\|q_k\|_2 + \varepsilon_k, \ \  \text{ where } \ \
    |\varepsilon_k| /\|q_k\|_2\rightarrow 0.
    \end{equation}
    Then
    \begin{align}
    \|pq_kp\|_2^2  & = \tau (pq_kpq_kp) \ge \tau(pq_kp_{11}q_kp)=\tau(p_{11}q_kpq_kp_{11})\ge\tau(p_{11}q_kp_{11}q_kp_{11})=\tau(p_{11}q_kp_{11}q_k) \notag\\
    & = \tau(p_{11} \tilde q_kp_{11}q_k)+ \tau(p_{11}(q_k-\tilde q_k)p_{11}q_k) \notag\\
    & =  \tau(p_{11} \tilde q_kp_{11} \tilde q_k) +  \tau(p_{11} \tilde q_kp_{11}(q_k -\tilde q_k))+  \tau((q_k-\tilde q_k)p_{11}q_kp_{11}) \notag\\
    & = \tau (p_{11} (\tilde q_k)^2)     +  \tau(p_{11} \tilde q_kp_{11}(q_k -\tilde q_k))+  \tau((q_k-\tilde q_k)p_{11}q_kp_{11})  \tag{\text{as $\tilde q_k\in \mathcal N'\cap \mathcal M$}}\\
    &\ge  \frac 1 m \|\tilde q_k\|_2^2- \|p_{11} \tilde q_kp_{11}\|_2\|q_k -\tilde q_k\|_2- \|q_k -\tilde q_k\|_2\|p_{11}q_kp_{11}\|_2 \tag{by Lemma \ref{may14Lemma5.75}}\\
    & \ge \frac 1 m (\|q_k\|_2+\varepsilon_k)^2 - 2(\|q_k\|_2+|\varepsilon_k|)\|q_k -\tilde q_k\|_2\notag\\
    &= \frac 1 m (\|q_k\|^2_2 + 2\varepsilon_k \|q_k\|_2+ \varepsilon_k^2) - 2(\|q_k\|_2+|\varepsilon_k|)\|q_k -\tilde q_k\|_2\notag
    \end{align}
    Hence, $(\ref{may15Equ4.1})$ and $(\ref{may15Equ4.2})$ guarantee
    \begin{equation*}
        \liminf_{k\rightarrow \infty} \frac {\|pq_kp\|_2^2}{\|q_k\|_2^2} \ge \frac 1
        m.
    \end{equation*}
    This completes the proof.
\end{proof}

Lemma \ref{may10Lem5.4}, Lemma \ref{may14Lemma5.7}, and Lemma \ref{may10Lemma5.10} are applied in the following Proposition \ref{may10Prop5.11}.

\begin{proposition}\label{may10Prop5.11}
    Let $\{q_k\}^{\infty}_{k=1}$ be  a sequence of nonzero projections in $\mathcal M$. Suppose that  $\mathcal N$, with $I_{\mathcal M}\in\mathcal N\subseteq \mathcal A (\{q_k\}^{\infty}_{k=1})$,   is a subfactor of type ${\rm I}_m$ for some positive integer $m$.  Let $\{p_{ij}\}_{ij=1}^m$ be a system of matrix units of $\mathcal N$. If $p$ is a projection in $\mathcal A (\{q_k\}^{\infty}_{k=1})$  with $p\ge p_{11}$, then there exists a   projection $e_k$ in $\mathcal M$  for each $k\ge 1$   such that
    \begin{enumerate}[label={\rm  (\roman*)}]
    \item \label{prop4.8.i} $e_k$ is a nonzero subprojection of $p$ when $k$ is large enough;
    \item  \label{prop4.8.ii}$\displaystyle \limsup_{k\rightarrow \infty} \frac {\|e_k\|_2}{\|q_k\|_2} \le 1 $;  and
    \item  \label{prop4.8.iii}
    $\displaystyle
    \lim_{k\rightarrow \infty} \frac {\|[pxp, e_k]\|_2}{\|e_k\|_2} =0, \qquad \text { for all } x\in \mathcal A (\{q_k\}^{\infty}_{k=1}).$
    \end{enumerate}
\end{proposition}

\begin{proof}
    By Lemma \ref{may14Lemma5.7},  there exists a subprojection $e_k$ of $p$ for each $k\ge 1$ such that
    \begin{equation}\label{may10Equ5.1}
        \lim_{k\rightarrow \infty} \frac {\|e_k-pq_kp\|_2}{\|q_k\|_2 }=0.
    \end{equation} Thus
    \begin{equation*}
        \limsup_{k\rightarrow \infty} \frac { \|e_k\|_2}{\|q_k\|_2}  \le \limsup_{k\rightarrow \infty}\frac {\|pq_kp\|_2+ \|pq_kp- e_k\|_2}{\|q_k\|_2} \le
        1,
    \end{equation*}
    which means that  \ref{prop4.8.ii} holds. By Lemma \ref{may10Lemma5.10},
    \begin{equation}\label{may10Equ5.2}
        \liminf_{k\rightarrow \infty} \frac { \|e_k\|_2}{\|q_k\|_2} \ge \liminf_{k\rightarrow \infty}\frac {\|pq_kp\|_2- \|pq_kp- e_k\|_2}{\|q_k\|_2}\ge \frac 1 {\sqrt m}.
    \end{equation}
    This means that   $e_k$ is nonzero when  $k$ is large enough. Thus,  \ref{prop4.8.i} is true. Moreover, for every $ x\in \mathcal A (\{q_k\}^{\infty}_{k=1})$, from Lemma \ref{may10Lem5.4}, the limit in (\ref{may10Equ5.1}), and the inequality in  (\ref{may10Equ5.2}), it follows that
    \begin{equation*}
        \begin{aligned}
            \limsup_{k\rightarrow \infty}\frac {\|[pxp, e_k]\|_2  }{\|e_k\|_2} & \le \limsup_{k\rightarrow \infty}  \frac {\|[pxp, pq_kp]\|_2 +   \|[pxp, e_k- pq_kp]\|_2  }{\|q_k\|_2} \cdot \frac {\|q_k\|_2}{\|e_k\|_2} \notag\\
            &\le \limsup_{k\rightarrow \infty} \frac {\|[pxp, pq_kp]\|_2 +  2 \|pxp\|\| e_k- pq_kp\|_2  }{\|q_k\|_2} \cdot \frac {\|q_k\|_2}{\|e_k\|_2}\notag\\&
            =0.
        \end{aligned}
    \end{equation*}
    This completes the proof.
\end{proof}

We are now ready to prove the main result in this section.  By $\mathcal{P}(\mathcal{M})$ we denote the set of all the projections in $\mathcal{M}$.

\begin{theorem}\label{prop5.3}
  Let $\mathcal M$ be a {non-$\Gamma$} type ${\rm II}_1$  factor with  trace $\tau$. Then there exist two self-adjoint elements $x_1$, $x_2$ in $\mathcal M$ and a positive number $\alpha>0$ such that
  \begin{equation*}
    \|[x_1, e]\|_2+\|[x_2, e]\|_2\ge \alpha \|e\|_2\|I-e\|_2, \qquad \forall \ e \in \P(\mathcal
    M).
  \end{equation*}
\end{theorem}

\begin{proof}
    By Proposition \ref{prop3.9-gamma}, there exist two self-adjoint elements $x_1$ and $x_2$ in $\mathcal M$ such that $W^*(x_1+i x_2)'\cap \mathcal M^{\omega}= \mathbb CI$. By Proposition \ref{prop3.9}, there exist an $\varepsilon>0$ and a nonzero projection $p$ in $\mathcal M$ such that for every nonzero subprojection $e$ of $p$,
    \begin{equation}\label{equ5.1}
    \text{either  \quad $\tau(e)>\varepsilon$ \quad or \quad  $\|[px_1p, e]\|_2+ \|[px_2p, e]\|_2>\varepsilon \|e\|_2 $.}
    \end{equation}

    If $p=I$, we claim that there exists an $\alpha$ such that $x_1$, $x_2$, and $\alpha$ have the desired property stated in the theorem.   Actually, assume, to the contrary, that there is no such an $\alpha$. Thus for each $\alpha=1/k$, there exists a projection  $q_k$  in $\mathcal M$ such that
    \begin{equation*}
        \|[x_1, q_k]\|_2+\|[x_2, q_k]\|_2 < \frac { \|q_k\|_2\|I-q_k\|_2} {k}.
    \end{equation*}
    Obviously, $0<\tau(q_k)<1$. Replacing $q_k$ by $I-q_k$ if needed, we assume that
    \begin{equation}\label{may16Equ5.6}
        \text {$0<\tau(q_k)\le 1/2$ \ for \ $k\ge 1$ \ \ and }  \ \  \lim_{k\to \infty} \frac { \|[x_1, q_k]\|_2+\|[x_2, q_k]\|_2}{\|q_k\|_2}=0.
    \end{equation}
    Notice that $W^*(x_1+i x_2)'\cap \mathcal M^{\omega}= \mathbb CI$.
    We have that
    \begin{equation}\label{may10Equ5.4}
        \lim_{k\to \omega} \|q_k\|_2 =0.
    \end{equation}
    Since we have assumed that $p=I$, we obtain that (\ref{may16Equ5.6}) and (\ref{may10Equ5.4}) contradict with (\ref{equ5.1}). This ends the proof in this case.

    Now we assume that $p\ne I$. Let $m\in\mathbb N$ be such that $2/m <\min\{\tau(p), \tau(I-p)\}$. Let $i_0$ be an integer such that $(i_0-1)/m\le \tau(p)<i_0/m$. Then $3\le i_0\le m-2$. Assume that $p$ and $px_1p$ are contained in a masa $\mathcal W$ of $\mathcal M$. Let $\{p_i\}_{i=1}^m$ be a family of mutually orthogonal   projections in $\mathcal W$ such that
    \begin{enumerate}[label={\rm (\alph*) }]
        \item   $p_1+\cdots+p_m=I$ and $\tau(p_i)=1/m$ for $1\le i\le m$;
        \item   $p_1, p_i, \ldots, p_{i_0-1}$ are subprojections of $p$;
        \item  $p_{i_0+1}, \ldots,p_{m-1}, p_m$ are subprojections of $I-p$.
    \end{enumerate}
    Let $\mathcal N$ be a type ${\rm I}_m$ subfactor of $\mathcal M$ with a system of matrix units $\{p_{ij}\}_{i,j=1}^m$ such that $p_{ii}=p_i$ for $1\le i\le m$. Let $\mathcal P=\mathcal N'\cap \mathcal M$. Then $\mathcal M  \cong \mathcal P\otimes \mathcal N$. By Lemma \ref{lemma2.2}, we obtain that $\mathcal P$ is also a type ${\rm II}_1$ factor without property $\Gamma$. Proposition \ref{prop3.9-gamma} implies there exist two self-adjoint elements $y_1, y_2$ in $\mathcal P$ such that $W^*(y_1, y_2)'\cap \mathcal P^{\omega}= \mathbb CI$. From Lemma \ref{lemma1}, it follows that
    \begin{equation}\label{equ5.2}
        W^*(y_1,y_2, \mathcal N)'\cap \mathcal M^\omega \cong  ( W^*(y_1,y_2)\otimes\mathcal N)'\cap (\mathcal P\otimes \mathcal N)^\omega \cong  W^*(y_1,y_2 )'\cap \mathcal P^\omega= \mathbb CI.
    \end{equation}
    Define $p_{i_0, 1}= pp_{i_0}$ and $p_{i_0, 2}:= (I-p)p_{i_0}$.

    Without loss of generality, we assume that $\|px_1p\|<1/10, \|y_1\|<1/10$ and $\|y_2\|<1/10$. Let $z_1$ and $z_2$ be elements in $\mathcal{M}$ defined as follows:
    \begin{equation*}
        \begin{aligned}
       z_1 & = \left (\sum_{j=1}^{i_0-1}   (2^jp_{j}+ p_{j}x_1p_{j})\right ) +   \left (2^{i_0} p_{i_0,1}+ p_{i_0,1}x_1p_{i_0,1} \right ) + (2^{i_0}+1)p_{i_0,2}   \\& \qquad \qquad \qquad   + \left (\sum_{j= i_0+1}^{m-2}    2^jp_{j} \right )    + (2^{m-1}p_{m-1}+p_{m-1}y_1)+ (2^{m}p_{m}+p_{m}y_2) \\
       z_2  &=  px_2p+ (p_{1m}+\cdots +p_{m-1, {m}}) + (p_{ m1}+\cdots +p_{m , {m-1}}) + p_m.
    \end{aligned}
    \end{equation*}
    Notice that
    \begin{equation}\label{e5.3}
        \{p_1,\ldots, p_{m},  p_{i_0,1},  p_{i_0,2}, p, px_1p, I-p\}\subseteq \mathcal W \text{ and } px_1p= (\sum_{j=1}^{i_0-1} p_ix_1p_i)+  p_{i_0,1}x_1p_{i_0,1} .
    \end{equation}
    By Lemma \ref{lemma5.2}, we have that
    \begin{equation*}
        \{p_1, \ldots, p_m, p_{i_0,1}, p_{i_0,2}, p, px_1p, p_{m-1}y_1, p_my_2\}\subseteq {C}^*(z_1).
    \end{equation*}
    From $p_i z_2p_m$ and $pz_2p$, we obtain that $\{p_{1m}, p_{2m},\ldots, p_{m {m}}, px_2p\}\subseteq C^*(z_1,z_2)$, whence
    \begin{equation}\label {equ5.3}
        \{p,   px_1p, px_2p, y_1, y_2,\mathcal N\}\subseteq C^*(z_1,z_2).
    \end{equation}
    Combining (\ref{equ5.2}) and  (\ref{equ5.3}), we have that
    \begin{equation}\label{equ5.6}
    \mathbb CI =  W^*(y_1,y_2, \mathcal N)'\cap \mathcal M^\omega \supseteq C^*(z_1,z_2)'\cap \mathcal M^\omega.
    \end{equation}
    To finish the proof of the theorem, it suffices to show the following claim.

    \begin{claim}
        There exists a constant $\alpha>0$ such that, for every projection $e$ in $\mathcal M$,
        \begin{equation*}
            \|[z_1,e]\|_2+ \|[z_2,e]\|_2\ge \alpha \|e\|_2\|I-e\|_2.
                    \end{equation*}
    \end{claim}

    \begin{proof}[Proof of the Claim]
    Assume, to the contrary, that for each $\alpha=1/k$ there exists a projection  $q_k$  in $\mathcal M$ such that
    \begin{equation*}
        \|[z_1, q_k]\|_2+\|[z_2, q_k]\|_2 < \frac { \|q_k\|_2\|I-q_k\|_2}
        {k}.
    \end{equation*}
    Obviously, $0<\tau(q_k)<1$. Replacing $q_k$ by $I-q_k$ if needed, we assume that $0<\tau(q_k)\le 1/2$. Then it follows that
    \begin{equation}\label{may15Equ5.9}
        \lim_{k\to \infty} \frac { \|[z_1, q_k]\|_2+\|[z_2, q_k]\|_2}{\|q_k\|_2}=0.
    \end{equation}
    Notice from (\ref{equ5.6}) that $W^*(z_1+i z_2)'\cap \mathcal M^{\omega}= \mathbb CI$. Therefore, we have that
    \begin{equation}\label{may10Equ5.10}
        \lim_{k\to \omega} \|q_k\|_2 =0.
    \end{equation}
    Let  $\mathcal A (\{q_k\}^{\infty}_{k=1})$ be as defined in Definition \ref{may14Def5.3}. From (\ref{equ5.3}) and (\ref{may15Equ5.9}),
    \begin{equation*}
        \{px_1p, px_2p, p,\mathcal N\}\subseteq \mathcal A
        (\{q_k\}^{\infty}_{k=1}).
    \end{equation*}
    Applying Proposition \ref{may10Prop5.11} to $\mathcal N$ and $p$, we can find  a   projection $e_k$ in $\mathcal M$  for each $k\ge 1$   such that
    \begin{enumerate}[label={\rm (\roman*)}]
        \item \label{thm4.9.i} $e_k $ is a nonzero subprojection of $p$ when $k$ is large enough.
        \item \label{thm4.9.ii} $\displaystyle \limsup_{k\to \infty} \frac {\|e_k\|_2}{\|q_k\|_2} \le 1 $  and
        \item  \label{thm4.9.iii} $\displaystyle \lim_{k\to \infty} \frac {\|[px_ip, e_k]\|_2}{\|e_k\|_2} =0, \qquad \text { for all } i=1,2.$
    \end{enumerate}
    Combining (\ref{may10Equ5.10}) and \ref{thm4.9.ii}, we have \begin{equation}\lim_{k\rightarrow \omega} \|e_k\|_2=0 \label{may15Equ4.12}\end{equation}   It is easy to check that \ref{thm4.9.i}, (\ref{may15Equ4.12}),  and \ref{thm4.9.iii} contradict  (\ref{equ5.1}). \phantom\qedhere
\end{proof}
This ends the proof of the claim and the proof of the theorem.
\end{proof}

 Combining with    Marrakchi's Proposition 2.2   in \cite{Ma},  Theorem \ref{prop5.3}  gives an operator with  spectral gap in a non-$\Gamma$ II$_1$ factor.
The result could be compared with Theorem 2.1 (c) in \cite{Connes}.
\begin{corollary}\label{may17Cor4.12}
    Let $\mathcal M$ be a type ${\rm II}_1$ factor with   trace $\tau$. Then the following statements are equivalent:
    \begin{enumerate}[label={\rm (\roman*)}]
    \item \label{cor4.11.i}  $\mathcal M$ is {non-$\Gamma$}, i.e., $\mathcal M$ fails to have property $\Gamma$.
    \item  \label{cor4.11.ii}  There exist  two self-adjoint elements   $x_1$ and $x_2$ in $\mathcal M$ and an $\alpha_1>0$ such that
    \begin{equation*}
        \|[x_1,y ]\|_2+ \|[x_2,y ]\|_2 \ge   \alpha_1 \|y-\tau(y)\|_2, \quad \text { for every } y\in \mathcal
        M.
    \end{equation*}
    \item  \label{cor4.11.iii} There exist an $x$ in $\mathcal M$  and an $\alpha_2>0$  such that
    \begin{equation*}
        \|[x, y]\|_2+\|[x^*, y]\|_2\ge \alpha_2 \|y-\tau(y)\|_2, \qquad \text{ for every } y \in\mathcal
        M.
    \end{equation*}
    \item  \label{cor4.11.iv}  There exist an $x$ in $\mathcal M$  and an $\alpha_3>0$  such that
    \begin{equation*}
        \|[x, y]\|_2\ge \alpha_3 \|y-\tau(y)\|_2, \qquad \text{ for every self-adjoint } y \in\mathcal
        M.
    \end{equation*}
    \item  \label{cor4.11.v}   There exist two unitary elements $u_1$ and $u_2$ in $\mathcal M$  and an $\alpha_4>0$  such that
    \begin{equation*}
        \|[u_1, y]\|_2 +\|[u_2, y]\|_2  \ge \alpha_4 \|y-\tau(y)\|_2, \qquad \text{ for every } y \in\mathcal
        M.
    \end{equation*}
    \end{enumerate}
\end{corollary}

\begin{proof}
    \ref{cor4.11.i} $ \Rightarrow $  \ref{cor4.11.ii}. Assume that $\mathcal M$ is {non-$\Gamma$}.  It follows from Theorem \ref{prop5.3} that there exist two self-adjoint elements $x_1$ and $x_2$ in $\mathcal M$ and a positive number $ \alpha$ such that
    \begin{equation*}
        \|[x_1,e ]\|_2+ \|[x_2,e ]\|_2 \ge  \alpha\|e\|_2\|I-e\|_2, \quad \text { for every  projection } e\in  \mathcal
        M.
    \end{equation*}
    By Proposition 2.2 in \cite{Ma}, there exists a positive number $\alpha_1$ such that
    \begin{equation*}
        \|[x_1,y ]\|_2+ \|[x_2,y ]\|_2 \ge   \alpha_1 \|y-\tau(y)\|_2, \quad \text { for every } y\in \mathcal
        M.
    \end{equation*}

    \ref{cor4.11.ii} $ \Rightarrow$  \ref{cor4.11.i}. It follows directly from the definition of property $\Gamma$.

    \ref{cor4.11.ii} $ \Leftrightarrow  $ \ref{cor4.11.iii}. Let $x=x_1+ix_2$ where $x_1, x_2$ are self-adjoint elements in $\mathcal M$. Then
    \begin{equation*}
        2(\|[x_1,y ]\|_2+  \|[x_2,y ]\|_2)\ge \|[x ,y ]\|_2+ \|[x^*,y ]\|_2 \ge \|[x_1,y ]\|_2+ \|[x_2,y ]\|_2   ,  \text { for every } y\in \mathcal
        M.
    \end{equation*}
    This means that the biconditional ``\ref{cor4.11.ii} $\Leftrightarrow$ \ref{cor4.11.iii}'' is obvious.

    \ref{cor4.11.iii} $ \Rightarrow $  \ref{cor4.11.iv}. It is trivial.

    \ref{cor4.11.iv} $ \Rightarrow  $ \ref{cor4.11.iii}. Assume that \ref{cor4.11.iv} is true. Let $y=y_1+iy_2$ be in $\mathcal M$ where $y_1, y_1$ are self-adjoint. Without loss of generality, we assume $\tau(y)=0$, thus $\tau(y_1)=\tau(y_2)=0$. Then
    \begin{equation*}
        \begin{aligned}
        \|y\|_2 &\le \|y_1\|_2+ \|y_2\|_2 \le \frac 1 {\alpha_3} \left ( \|[x, y_1]\|_2+  \|[x, y_2]\|_2  \right ) \\ &\le \frac 1 {\alpha_3} \left ( \|[x, y]\|_2+  \|[x, y^*]\|_2  \right )=\frac 1 \alpha_3 \left ( \|[x, y]\|_2+  \|[x^*, y]\|_2  \right ).
        \end{aligned}
    \end{equation*}
    i.e., \ref{cor4.11.iii} is true.

    \ref{cor4.11.ii} $ \Rightarrow  $ \ref{cor4.11.v}. Assume that \ref{cor4.11.ii} is true. Let $\lambda>\max\{\|x_1\|, \|x_2\|\}$ and
    \begin{equation*}
        u_1= \frac {x_1}{\lambda } + i \sqrt {I- \frac {x_1^2} {\lambda^2}} \qquad \text { and } \qquad u_2= \frac {x_2}{\lambda } + i \sqrt {I- \frac {x_2^2}
        {\lambda^2}}.
    \end{equation*}
    Then $u_1$, $u_2$ are unitary elements in $\mathcal M$ such that $u_i^*+u_i=2x_i/\lambda$ for $i=1,2$. Hence
    \begin{equation*}
        \begin{aligned}
            \|[u_1, y]\|_2 +\|[u_2, y]\|_2  & = \frac {\|[u_1, y]\|_2 + \|[u_1^*, y]\|_2}{2} + \frac {\|[u_2, y]\|_2 + \|[u_2^*, y]\|_2}{2}\\
            & \ge \frac 1 \lambda \left ( \|[x_1, y]\|_2 + \|[x_2, y]\|_2\right ) \\
            & \ge \frac \alpha \lambda \|y-\tau(y)\|_2, \qquad \forall \ y\in \mathcal M.
        \end{aligned}
    \end{equation*} Thus, \ref{cor4.11.v} is true.

    \ref{cor4.11.v} $ \Rightarrow  $ \ref{cor4.11.iv}. Assume that \ref{cor4.11.v} is true. From Theorem 5.2.5 of \cite{Kadison1}, we have $u_1=e^{ix_1}$ and $u_2=e^{ix_2}$ for some positive elements $x_1, x_2$ in $\mathcal M$. We show that there exists an $\alpha'>0$ such that
    \begin{equation*}
        \|[x_1+ix_2, y]\|_2\ge \alpha' \|y-\tau(y)\|_2, \qquad \text{ for every self-adjoint } y \in\mathcal
        M.
    \end{equation*}
    Assume, to the contrary, that for any $\alpha'=1/k$ there exists a self-adjoint element $y_k$ such that (1) $\tau(y_k)=0$; (2) $\|y_k\|_2=1$; and (3) $\|[x_1+ix_2, y_k]\|_2< 1/k$. Define
    \begin{equation*}
        \mathcal B(\{y_k\}^{\infty}_{k=1}) = \{x\in\mathcal M : \lim_{k\to \infty} \|[x,
        y_k]\|_2=0\}.
    \end{equation*}
    Similar to Lemma \ref{may18Lemma4.2}, we obtain that $\mathcal B(\{y_k\}^{\infty}_{k=1})$ is a unital $C^*$-algebra  containing $x_1+ix_2$. Thus $u_1,u_2\in \mathcal B(\{y_k\}^{\infty}_{k=1})$, which contradicts Assumption \ref{cor4.11.v}.
\end{proof}


\section{Reducible operators in non-\texorpdfstring{$\Gamma$} \ \ type \texorpdfstring{${\rm II}_1$} \ \ factors}

In this section, let $\mathcal M$ be a type ${\rm II}_1$ factor with trace $\tau$. 

\begin{definition}\label{red-opts}
    An element $x\in\mathcal M$ is reducible if there is a nontrivial projection $p\in\mathcal M$ such that $xp=px$, equivalently, $W^*(x)'\cap \mathcal M\ne \mathbb CI$. If an element $x\in\mathcal M$ is not reducible, then $x$ is irreducible, equivalently, $W^*(x)'\cap \mathcal M=\mathbb CI$. The set of all the reducible operators in $\mathcal M$ is denoted by $\operatorname{Red}(\mathcal M)$ and the operator-norm closure of $\operatorname{Red}(\mathcal M)$ is denoted by $\overline{\operatorname{Red}(\mathcal M)}^{\|\cdot\|}$.
\end{definition}

\begin{proposition}\label{June28Prop5.2}
    Suppose that $\mathcal M$ is a type ${\rm II}_1$ factor with separable predual. Then $\operatorname{Red}(\mathcal M)$ is not operator norm closed in $\mathcal M$. In particular, $\operatorname{Red}(\mathcal M)\ne \mathcal M$.
\end{proposition}

\begin{proof}
    As a special case of Corollary 4.1 of \cite{Pop}, there exists an irreducible, hyperfinite subfactor $\R$ of $\mathcal{M}$, i.e., $\R^{\prime}\cap\M=\mathbb{C} I$.

    Notice that there exists a unital CAR subalgebra $\A$ of $\R$ such that $\R$ is the ${\rm weak}^{*}$-closure of $\A$. The reader is referred to Example III.2.4 of \cite{Davidson} for the definition of CAR algebras. Thus there exists an increasing sequence $\{\A_{n}\}^{\infty}_{n=1}$ of full matrix algebras such that
    \begin{enumerate}
        \item $\A_{n}$ is $\ast$-isomorphic to $M_{2^n}(\CCC)$  for each $n\geq 1$;
        \item $\cup^{\infty}_{n=1}\A_{n}$ is dense in $\A$ in the operator norm.
    \end{enumerate}
    In terms of the main theorem of \cite{Topping}, there exists a single generator $a\in \A$. It follows that $W^*(a)=\R$. This entails that $a$ is irreducible in $\mathcal{M}$, or $a\notin \operatorname{Red}(\mathcal M)$. The fact $\A$ is a CAR algebra implies that there exists a sequence $\{a_n\}^{\infty}_{n=1}$ of operators in $\A$ with $a_n\in \A_n$ for each $n\geq 1$ such that
    \begin{equation*}
        \lim_{n\to \infty}\Vert a_n-a \Vert=0.
    \end{equation*}
    That $a_n\in \A_n$ entails that $a_n$ is reducible in $\mathcal{M}$ for every $n\geq 1$. Hence $a\in \overline{\operatorname{Red}(\mathcal M)}^{\|\cdot \|}$.  Thus $\operatorname{Red}(\mathcal M)$ is not closed in $\mathcal M$ in the operator norm and $\operatorname{Red}(\mathcal M)\ne \mathcal M$.
\end{proof}

Suppose $\mathcal N$ is a separable type ${\rm II}_1$ factor with property $\Gamma$. It is straightforward to see that $\operatorname{Red}(\mathcal N^\omega)=\mathcal N^\omega$. In fact, if $(x_n)_\omega\in \mathcal N^\omega$, then there exists a sequence  $\{q_n\}_{n=1}^\infty$ of projections in $\mathcal N$ with $\tau(q_n)=1/2$ and $\|[x_n, q_n]\|_2\le 1/n$ for each $n\ge 1$. Then $(q_n)_\omega$ is a nontrivial projection in $\mathcal N^\omega$ that commutes with $(x_n)_\omega$. Hence,  $\operatorname{Red}(\mathcal N^\omega)=\mathcal N^\omega$.
Here we present another type of examples of nonseparable type ${\rm II}_1$ factors $\mathcal M$ such that $\operatorname{Red}(\mathcal M)=\mathcal M$.

\begin{example}\label{June10Example5.2}
    Let $\Lambda$ be an uncountable index set and $(\mathcal{M}_{\lambda},\tau_{\lambda})$ a type ${\rm II}_1$ factor with trace $\tau_{\lambda}$ acting on $L^2(\mathcal{M}_{\lambda},\tau_{\lambda})$ for each $\lambda\in \Lambda$. Write $\mathcal{M} = \overline{\bigotimes}_{\lambda \in \Lambda}(\mathcal{M}_{\lambda},\tau_{\lambda})$ to be the tensor product von Neumann algebra. 
    It is routine to verify that $\mathcal{M}$ is a type ${\rm II}_1$ factor with trace $\tau$, where $\tau$ is induced by $\{\tau_{\lambda}\}_{\lambda\in\Lambda}$. Naturally, we can view $\mathcal{M}$ as a subset of $L^2(\mathcal{M},\tau)$.  Correspondingly, the underlying Hilbert space $L^2(\mathcal{M},\tau)$ can be viewed as $\bigotimes_{\lambda \in \Lambda} L^2(\mathcal{M}_{\lambda},\tau_{\lambda})$ (see Definition 1.6 from Chapter XIV of \cite{Takesaki3}, where the method works for an uncountable tensor product by taking $\hat{I}_{\lambda}$ in each $L^2(\mathcal{M}_{\lambda},\tau_{\lambda})$ as reference vectors).

    Suppose that $a$ is an element in $\mathcal{M}$. From $a\in L^2(\mathcal{M},\tau)$, there is a countable subset $\Lambda_1$ of $\Lambda$ such that $a \in \bigotimes_{\lambda \in \Lambda_1} L^2(\mathcal{M}_{\lambda}, \tau_{\lambda})$.  It follows that $a \in \overline{\bigotimes}_{\lambda \in \Lambda_1} (\mathcal{M}_{\lambda}, \tau_{\lambda})$. Thus, there is an index $\alpha \in \Lambda \setminus \Lambda_1$ such that $(\mathcal{M}_{\alpha},\tau_{\alpha})$ (as a nontrivial subfactor of $\mathcal{M}$) is contained in the relative commutant of $W^*(a)$ in $(\mathcal{M},\tau)$. Therefore, $a \in \operatorname{Red}(\mathcal{M})$, which means $\operatorname{Red}(\mathcal{M}) = \mathcal{M}$. 

    In contrast to each $\mathcal N$ possessing property $\Gamma$ in the paragraph preceding this example, we can set every $(\mathcal{M}_{\lambda},\tau_{\lambda})$ to be the free group factor $L(\mathbb{F}_2)$ for all $\lambda \in \Lambda$. 
    It is also worth noting that the type ${\rm II}_1$ factor $\mathcal{M}$ constructed here has property $\Gamma$.
\end{example}


In \Cref{June28Prop5.2}, we obtain that the set of reducible operators in each separable type ${\rm II}_1$ factor is not operator norm closed. \Cref{June10Example5.2} provides us  a family of non-separable type ${\rm II}_1$ factors $\mathcal M$ with property $\Gamma$ satisfying $\operatorname{Red}(\mathcal M) = \mathcal M$. In the next result, we will show that, in a  (separable or nonseparable) non-$\Gamma$ type ${\rm II}_1$ factor, the set of reducible operators fails to be closed in the operator norm topology.

\begin{proposition}\label{June28Prop5.11}  Let $\mathcal N$ be a separable type ${\rm II}_1$ factor and $\mathcal M$ a non-$\Gamma$ type ${\rm II}_1$ factor.
    \begin{enumerate}[label={\rm (\roman*)}] 
        \item \label{prop5.5.i}  If $x$ is an element in $\mathcal N$ satisfying $W^*(x)=\mathcal N$, then there exists an element $y$ in $\mathcal N$ such that
        \begin{equation*}
            W^*(y)=\mathcal N \quad \text{ and } \quad y\in\overline{\operatorname{Red}(\mathcal N)}^{\|\cdot\|}.
        \end{equation*}
        \item \label{prop5.5.ii}  $\operatorname{Red}(\mathcal M)$ is not operator norm closed in $\mathcal M$.
    \end{enumerate}
\end{proposition}

\begin{proof}
    \ref{prop5.5.i}. Assume that $x=x_1+ix_2$ for some self-adjoint elements $x_1, x_2$ in $\mathcal N$. Suppose that $x_1$ is contained in  a masa $\mathcal A$ in
    $\mathcal N$. Assume that $y_1$ is  a self-adjoint
    generator of $\mathcal A$. Let $\{p_n\}_{n=1}^\infty$ be an
    increasing sequence of nontrivial projections in $\mathcal A$ such that
    $\lim_{n\rightarrow\infty} \|I-p_n\|_2=0$. Define
    $$ z_m=\sum_{n=1}^m \frac {p_nx_2p_n}{2^n}\quad  \ \text {   for each $m\ge 1$  \ \  and }
    \qquad y_2= \sum_{n=1}^\infty \frac {p_nx_2p_n}{2^n}.
    $$ Thus $y_2$ is a self-adjoint element in
    $\mathcal N$ such that $  \lim_{m\rightarrow \infty} \|y_2-z_m\|=0$.
    From the fact that $z_mp_m =z_m=  p_m z_m$, it follows that
    $y_1+iz_m\in\operatorname{Red}(\mathcal N)$, whence
    $y_1+iy_2\in\overline{\operatorname{Red}(\mathcal N)}^{\|\cdot\|}$.
    
    Define $y=y_1+iy_2$. We next show that $W^*(y)=\mathcal N$. In fact, by the choice
    of $y_1$, we have that $x_1$ and  $\{p_n\}_{n=1}^\infty$ are in $\mathcal A\subseteq W^*(y)$.  By the
    construction of $y_2$, we have $$p_1y_2p_1=\sum_{n=1}^\infty \frac {p_1x_2p_1}{2^n} = p_1x_2p_1\in
    W^*(y_1,y_2),$$ and $$ \begin{aligned}
    p_{m+1}y_2p_{m+1}& =\left ( \sum_{n=1}^m \frac {p_nx_2p_n}{2^n} \right )+\left ( \sum_{n=m+1}^\infty \frac {p_{m+1}x_2p_{m+1}}{2^n} \right ) \\& = \left ( \sum_{n=1}^m \frac {p_nx_2p_n}{2^n} \right )+\frac {p_{m+1}x_2p_{m+1}}{2^m} \in
    W^*(y_1,y_2), \quad \text { for $m\ge 1$}.\end{aligned}
    $$ Therefore
    we obtain that $p_mx_2p_m\in W^*(y ) $ for $m\ge 1$. This
    implies that $x_2\in W^*(y ),$ as $\lim_{m\rightarrow\infty} \|I-p_m\|_2=0$. It follows that $W^*(y )=\mathcal
    N$.
    
    \ref{prop5.5.ii}. By Proposition \ref{prop3.9-gamma}, there exists an element  $x $ in $\mathcal M$ such that $W^*(x )'\cap \mathcal M^\omega=\mathbb CI$, so  $W^*(x )'\cap \mathcal M=\mathbb CI$. Define $\mathcal N=W^*(x )$. Then $\mathcal N$ is an irreducible subfactor of $\mathcal M$. By part (i), there  exists
    an operator $y $ in $\mathcal N$ such that
    $W^*(y )=\mathcal N$ and
    $y \in\overline{\operatorname{Red}(\mathcal N)}^{\|\cdot\|}$. It follows that $y $ is an irreducible operator in $\mathcal M$ with $y \in  \overline{\operatorname{Red}(\mathcal M)}^{\|\cdot\|}$. This finishes the proof of \ref{prop5.5.ii}.
\end{proof}

Recall that
\begin{equation*}
    \ell^{\infty} ( \mathcal M) =\{(a_n)_n \ : \ \forall \ n\in
\mathbb N, \  a_n\in\mathcal M \ \text { and } \ \sup_{n\in\mathbb
N} \|a_n\|<\infty\}.
\end{equation*}
and
\begin{equation*}
    c_0(\mathcal M) =\{(a_n)_n\in \ell^{\infty} ( \mathcal M) :
\lim_{n\rightarrow \infty} \|a_n\|=0\}.
\end{equation*}
Then $c_0(\mathcal M)$ is a norm closed two sided ideal of $\ell^{\infty} ( \mathcal
M)$ and
\begin{equation*}
    \ell^{\infty} ( \mathcal M)/c_0(\mathcal M)
\end{equation*}
is also a unital $C^*$-algebra. An element in $\ell^{\infty} ( \mathcal M)/c_0(\mathcal M)$  is denoted by $[(a_n)_n]$, if no confusion arises. Moreover, there is a natural embedding from $\mathcal M$ into $\ell^{\infty} ( \mathcal M)/c_0(\mathcal M)$ by sending $a$ in $\mathcal M$ to $[(a)_n]$ in $\ell^{\infty} ( \mathcal M)/c_0(\mathcal M)$. So we view $\mathcal M \subseteq  \ell^{\infty} ( \mathcal M)/c_0(\mathcal M)$.

To proceed, we need to prepare a well-known technical result for $ C^* $-algebras, which is inspired by Exercise 2.7 from \cite{Ror}. 
For completeness, we include a proof. 

\begin{lemma}\label{lemma4.4}
    Let $x$ be a self-adjoint element in $\mathcal M$ such that $\|x-x^2\|<1/4$. Then there is a projection $p\in C^*(x)$ such that $\|x-p\|\le \sqrt {\|x-x^2\|}$.
\end{lemma}

\begin{proof}
    Assume that $\Vert x^2- x \Vert = \varepsilon < 1/4$. As an application of the spectral theorem for normal operators, it follows that
    \begin{equation}
      \sigma(x)\subseteq[-\sqrt{\varepsilon},\sqrt{\varepsilon}]\cup [1-\sqrt{\varepsilon},1+\sqrt{\varepsilon}]. \nonumber
    \end{equation}
    Define a continuous function $f$ on $\sigma(x)$ in the following form:
    \begin{equation}\label{f(t)}
      f(t):=\begin{cases}
          0, & t\in\sigma(x)\cap [-\sqrt{\varepsilon},\sqrt{\varepsilon}]\\
          1, & t\in\sigma(x)\cap [1-\sqrt{\varepsilon},1+\sqrt{\varepsilon}].
      \end{cases}  \nonumber
    \end{equation}
    By the spectral mapping theorem for normal operators, we have $\Vert f(x)- x \Vert \leq \sqrt{\varepsilon}$. Note that $f(x)$ is a projection in $ C^* (x)$. This ends the proof.
\end{proof}

\begin{proposition}\label{may15Prop5.5}
    Let $x$ be an element in a type ${\rm II}_1$ factor $\mathcal M$. The following statements are equivalent:
    \begin{enumerate}[label={\rm (\roman*)}]
    \item \label{prop5.7.i} $x\in \overline{\operatorname{Red}(\mathcal M)}^{\|\cdot\|}$.
    \item \label{prop5.7.ii} There exists a sequence $\{p_n\}_{n=1}^\infty$ of projections in $\mathcal M$ such that
    \begin{equation*}
         0<\tau(p_n)\le 1/2, \ \forall \ n\ge 1, \mbox{ and } \lim_{n\to \infty} \|[x,
         p_n]\|=0.
    \end{equation*}
    \item \label{prop5.7.iii} There exists a sequence $\{p_n\}_{n=1}^\infty$ of projections in $\mathcal M$ such that
    \begin{equation*}
        \text{ $0<\tau(p_n)\le 1/2$, $\forall \ n\ge 1$, and $\lim_{n\to \infty} \|[y,p_n]\|=0$, $\forall \ y\in {C}^*(x)$.}
    \end{equation*}
    \item \label{prop5.7.iv} $C^*(x)'\cap \left (\ell^{\infty} ( \mathcal M)/c_0(\mathcal M)\right )$ has a projection not contained in the center of  $ \ell^{\infty} ( \mathcal M)/c_0(\mathcal M) $.
    \end{enumerate}
\end{proposition}

\begin{proof}
    \ref{prop5.7.i} $\Rightarrow$ \ref{prop5.7.ii}. Suppose that $x\in \overline{\operatorname{Red}(\mathcal M)}^{\|\cdot\|}$. Then there exists a sequence $\{x_n\}^{\infty}_{n=1}$ of operators in $\operatorname{Red}(\mathcal M)$ such that $\lim_{n\to \infty}\Vert x_n-x\Vert=0$. Thus, for each reducible operator $x_n$ in $\mathcal{M}$ there exists a nontrivial projection $p_n$ in $\mathcal{M}$ such that $p_n x_n=x_n p_n$ for every $n\in\NNN$. Replacing $p_n$ by $I-p_n$ if $\tau(p_n)> 1/2$, we assume that $0<\tau(p_n)\leq 1/2$.  Note that
    \begin{equation*}
        \Vert x p_n-p_n x \Vert = \Vert x p_n-p_n x_n + x_n p_n-p_n x\Vert \leq 2 \Vert x_n-x\Vert
    \end{equation*}
    This completes the proof of the implication \ref{prop5.7.i} $\Rightarrow$ \ref{prop5.7.ii}.

    \ref{prop5.7.ii} $\Rightarrow$ \ref{prop5.7.i}. Assume that $\{p_n\}^{\infty}_{n=1}$ is a sequence of nontrivial projections in $\mathcal{M}$ such that $\lim_{n\to \infty} \|[x,p_n]\|=0$. Define $x_n=p_n x p_n + (I-p_n) x (I-p_n)$. Since $p_n$ is nontrivial, we obtain that $x_n$ is reducible in $\mathcal{M}$. Thus  $\Vert x_n-x\Vert=\Vert x p_n-p_n x \Vert$.
    This completes the proof of the implication \ref{prop5.7.ii} $\Rightarrow$ \ref{prop5.7.i}.

    \ref{prop5.7.iii} $\Leftrightarrow$ \ref{prop5.7.ii}. Note that the implication of \ref{prop5.7.iii} $\Rightarrow$ \ref{prop5.7.ii} is trivial. Assume that \ref{prop5.7.ii} holds. Define $\A$ to be a set in the following form:
    \begin{equation*}
        \A := \{y\in \M: \lim_{n \to \infty}\Vert y p_n-p_n y \Vert
        =0\}.
    \end{equation*}
    By a proof similar to that of \Cref{may18Lemma4.2}, we obtain that $\A$ is a unital $ C^* $-algebra containing $x$. Thus, $ C^* (x)\subseteq \A$. This completes the proof of the implication \ref{prop5.7.ii} $\Rightarrow$ \ref{prop5.7.iii}.

    \ref{prop5.7.iii} $\Rightarrow$ \ref{prop5.7.iv} $\Rightarrow$ \ref{prop5.7.ii}.  The implication \ref{prop5.7.iii} $\Rightarrow$ \ref{prop5.7.iv} is trivial. Now assume that \ref{prop5.7.iv} is true. Let $[(p_n)_n]$ be in $C^*(x)'\cap \left (\ell^{\infty} ( \mathcal M)/c_0(\mathcal M)\right )$ but not contained in the center of  $ \ell^{\infty} ( \mathcal M)/c_0(\mathcal M) $. From Lemma \ref{lemma4.4}, we can assume that each $p_n$ is a projection in $\mathcal M$. Note that $[(p_n)_n]$ is not in the center of  $ \ell^{\infty} ( \mathcal M)/c_0(\mathcal M) $. There must be an increasing sequence $\{n_k\}_{k=1}^\infty$ of positive integers such that $p_{n_k}$ is non-trivial for each $k$. Apparently,  $\lim_{k\to \infty} \|[x,p_{n_k}]\|=0.$ Thus  \ref{prop5.7.ii} is true.
\end{proof}

The following lemmas are useful.

\begin{lemma}\label{may15Lemma5.6}
    Let $x_1$ and $x_2$ be self-adjoint elements in $\mathcal M$. If there exist a positive number $\alpha>0$ and a projection $p\in \mathcal M$ with $\tau(p)>0$, satisfying
    \begin{equation}
        \|[x_1,p]\|_2+ \|[x_2,p]\|_2\ge \alpha \|p\|_2, \label{eqn-0}
    \end{equation}
    then
    \begin{equation*}
        \|[x_1,p]\| + \|[x_2,p]\| \ge \frac \alpha  {\sqrt 2}.
    \end{equation*}
\end{lemma}

\begin{proof}
    By the definition of the operator norm, we have
    \begin{equation}\label{eqn-1}
        \Vert x_ip-px_i \Vert \geq \Vert (x_ip-px_i)\frac{p}{\Vert p \Vert_2}
         \Vert_2= \frac{1}{\Vert p \Vert_2} \Vert (I-p)x_ip \Vert_2 \ \mbox{ for } \ i=1,2.
    \end{equation}
    Since the equality $\Vert (I-p)x_i p \Vert^2_2=\Vert p x_i (I-p) \Vert^2_2$ holds for $i=1,2$, it follows that
    \begin{equation}\label{eqn-2}
    \begin{aligned}
        \Vert x_ip-px_i \Vert^2_2 & =  \Vert (I-p)x_i p \Vert^2_2+\Vert p x_i (I-p) \Vert^2_2 \\
        & =  2 \Vert (I-p)x_i p \Vert^2_2.
    \end{aligned}
    \end{equation}
    Inequality (\ref{eqn-1}) and equality (\ref{eqn-2}) entail that
    \begin{equation}\label{eqn-3}
        \Vert x_ip-px_i \Vert \geq \frac{1}{\sqrt{2}\Vert p \Vert_2} \Vert x_ip-px_i \Vert_2  \ \mbox{ for } \ i=1,2.
    \end{equation}
    Inequalities (\ref{eqn-0}) and (\ref{eqn-3}) guarantee that
    \begin{equation*}
        \Vert x_1 p-px_1 \Vert+\Vert x_2 p-px_2 \Vert \geq \frac {\alpha} {{\sqrt
        2}}.
    \end{equation*}
    This completes the proof.
\end{proof}

\begin{lemma}\label{lemma4.7}
    Let $u_1$ and $u_2$ be unitary elements in $\mathcal M$. If there exist a positive number $\alpha>0$ and a projection $p\in \mathcal M$ with $\tau(p)>0$, satisfying
    \begin{equation}\label{eqn-5}
        \|[u_1,p]\|_2+ \|[u_2,p]\|_2\ge \alpha \|p\|_2,
    \end{equation}
  then
  \begin{equation*}
    \|[u_1,p]\| + \|[u_2,p]\| \ge \frac \alpha  {\sqrt 2}.
  \end{equation*}
\end{lemma}

\begin{proof}
    Note that, for $i=1,2$, the equality
    \begin{equation*}
        \Vert p u_i p \Vert^2_2+\Vert (I-p)u_ip \Vert^2_2=\Vert u_ip \Vert^2_2=\tau(p)=\Vert p u_i \Vert^2_2 = \Vert p u_i p \Vert^2_2+\Vert pu_i(I-p)
        \Vert^2_2
    \end{equation*}
    yields that
    \begin{equation*}
        \Vert (I-p)u_ip \Vert^2_2=\Vert p u_i (I-p)\Vert^2_2.
    \end{equation*}
    The rest of the proof is the same as that of \Cref{may15Lemma5.6}.
    This completes the proof.
\end{proof}

In the following result, we show that, in a non-$\Gamma$ type ${\rm II}_1$ factor, there always exist operators not in the operator norm closure of the reducible ones.

\begin{theorem}\label{may14Thm5.14}
Let $\mathcal M$ be a non-$\Gamma$ type ${\rm II}_1$ factor. Then $\overline{\operatorname{Red}(\mathcal M)}^{\|\cdot\|}\ne \mathcal M$.
\end{theorem}

\begin{proof}
    It follows from Theorem \ref{prop5.3} that there exist two self-adjoint elements $x_1$ and $x_2$ in $\mathcal M$ and a positive number $ \alpha$ such that
    \begin{equation*}
        \|[x_1,e ]\|_2+ \|[x_2,e ]\|_2 \ge  \alpha \|e\|_2\|I-e\|_2, \quad \text { for every projection  } e\in \mathcal
        M.
    \end{equation*}
    By Lemma \ref{may15Lemma5.6}, we obtain that
    \begin{equation*}
        \|[x_1,e ]\|+ \|[x_2,e ]\| \ge \alpha/2 , \  \text { for every projection  } e\in \mathcal M \text{ with } 0<\tau(e)\le
        1/2.
    \end{equation*}
    From Proposition \ref{may15Prop5.5}, it follows that  $x_1+ix_2\notin \overline{\operatorname{Red}(\mathcal M)}^{\|\cdot\|}$, i.e., $\overline{\operatorname{Red}(\mathcal M)}^{\|\cdot\|}\ne \mathcal M$.
\end{proof}

In the next example, we construct an explicit operator in the free group factor $L(\mathbb{F}_2)$ such that it is not in the operator norm closure of the reducible ones.

\begin{example}
    Denote by $u_1$ and $u_2$ the unitary operators in $L(\mathbb{F}_2)$ corresponding to the two generators of $\mathbb{F}_2$. 
    For every element $y$ in $L(\FFF_2)$, from Theorem $6.7.8$ in \cite{Kadison2} it follows that
    \begin{equation}\label{u_1&u_2}
        \Vert [u_1,y] \Vert_2+\Vert [u_2,y] \Vert_2 \geq \frac{\Vert y - \tau(y) \Vert_2}{12}.
    \end{equation}
    Combining it with Inequality $(\ref{u_1&u_2})$ and Lemma $\ref{lemma4.7}$, for a  nontrivial projection $p$ in $L(\FFF_2)$  with $0<\tau(p) \leq 1/2$, it follows that
    \begin{equation}\label{u_i&p-5}
        \Vert u_1 p - p u_1 \Vert+\Vert u_2 p - p u_2 \Vert \geq \frac{1}{24}.
    \end{equation}
    By Theorem $5.2.5$ of \cite{Kadison1}, there are two positive operators $a_1$ and $a_2$ in $ L(\FFF_2)$ such that
    \begin{equation*}
        u_1=e^{ ia_1}, \quad u_2=e^{ ia_2}, \quad \mbox{ and } \quad  \Vert a_k\Vert\leq 2\pi,\quad \mbox{ for }
        k=1,2.
    \end{equation*}
    We claim that, for $x=a_1+ia_2$, there exists an $\varepsilon_0>0$ such that
    \begin{equation}\label{u_i&p-6}
        \Vert xp-px \Vert \geq \varepsilon_0 \ \mbox{ for every nontrival projection }\ p\in \M.
    \end{equation}
    Otherwise, there is a sequence of projections $\{p_n\}^{\infty}_{n=1}$ in $ L(\FFF_2)$ such that the inequalities
    \begin{equation*}
        \Vert x p_n-p_n x \Vert\leq \frac{1}{n} \quad \mbox{ and } \quad 0< \tau(p_n)\leq
        \frac{1}{2}
    \end{equation*}
    hold for each integer $n\geq 1$.
    This yields that $\lim_{n\rightarrow \infty}\Vert u_k p_n-p_n u_k \Vert=0$ for $k=1,2$, which contradicts the inequality in $(\ref{u_i&p-5})$. This ends the proof of $(\ref{u_i&p-6})$.
\end{example}

\section{Nowhere-dense-property of reducible operators in non-\texorpdfstring{$\Gamma$} \ \ type \texorpdfstring{${\rm II}_1$} \ \ factors}

The goal of this section is to prove that the set of reducible operators in each non-$\Gamma$ type ${\rm II}_1$ factor is nowhere dense, in the operator norm topology. For this purpose, we introduce the following definition.

\begin{definition}
Let $\mathcal B$ be a unital $C^*$-algebra with an identity $I_{\mathcal B}$,  and let $\mathcal A\subseteq \mathcal B$ be a $C^*$-subalgebra containing $I_{\mathcal B}$. Then an element $x\in \mathcal A$ is called \textbf{reducible in }$\mathcal B$ if there exists a projection $p$ in $\mathcal B\setminus \mathcal Z(\mathcal B)$ such that $xp=px$, where $\mathcal Z(\mathcal B)$ is the center of $\mathcal B$.

Define  $\operatorname{Red}(\mathcal A:\mathcal B)$ to be the set of all these elements of $\mathcal A$ that are reducible in $\mathcal B$, i.e.,
\begin{equation*}
    \operatorname{Red}(\mathcal A:\mathcal B) =\{x\in \mathcal A: \text{ there exists a   projection $p$ in  $\mathcal B\setminus \mathcal Z(\mathcal B)$ such that }
    xp=px\}.
\end{equation*}
\end{definition}

\begin{remark}
    Let $\mathcal M$ be a type ${\rm II}_1$ factor. If $\mathcal A=\mathcal B=\mathcal M$, then $\operatorname{Red}(\mathcal M:\mathcal M)=\operatorname{Red}(\mathcal M)$, where $\operatorname{Red}(\mathcal M)$ is introduced in Definition $\ref{red-opts}$.
\end{remark}

\begin{remark}
    Let $\mathcal{H}$ be an infinite-dimensional, complex, separable Hilbert space and $ \mathcal B(\mathcal H) $   the algebra of all the bounded linear operators on $\mathcal{H}$. Then, as an application of Voiculescu's non-commutative Weyl-von Neumann theorem in \cite{Voi2}, it is well-known that
    \begin{equation*}
        \operatorname{Red}\Big( \mathcal B(\mathcal H) :  {\ell^{\infty} \big( \mathcal B(\mathcal H) \big)}/{c_0\big( \mathcal B(\mathcal H) }\big)\Big)= \mathcal B(\mathcal H) .
    \end{equation*}
\end{remark}

\begin{remark}\label{M:M-omega}
    Let $\mathcal M$ be a type ${\rm II}_1$ factor and $\omega$ a free ultrafilter on $\mathbb N$. Then, by Proposition $\ref{prop3.9-gamma}$, $\mathcal M$ has property $\Gamma$ if and only if
    \begin{equation*}
        \operatorname{Red}(\mathcal M: \mathcal M^\omega) =\mathcal
        M.
    \end{equation*}
\end{remark}

\begin{proposition}\label{N:B}
    Let $\mathcal B$ be a unital $C^*$-algebra with an identity $I_{\mathcal B}$, and let $\mathcal M$ be a factor of type ${\rm II}_1$ such that $I_{\mathcal B}\in \mathcal M\subseteq \mathcal B$. If $ p\mathcal Mp\setminus\operatorname{Red}(p\mathcal Mp: p\mathcal Bp)\ne \emptyset$   for any nonzero projection $p$ in $\mathcal M$, then  $\mathcal M\setminus \operatorname{Red}(\mathcal M: \mathcal B)$ is dense in $\mathcal M$ in the operator norm topology. 
\end{proposition}

\begin{proof}
    Suppose that $x = x_1+ i x_2$ is an element in $\mathcal{M}$, where $x_1$ and $x_2$ are two self-adjoint operators in $\mathcal{M}$. Let $\varepsilon$ be a positive number.

    By spectral theory for   $x_1$, there exist an $m\in\mathbb N$, an orthogonal family of projections $\{q_{k}\}^{m}_{k=1}$ in $\mathcal{M}$ with sum $I$ and a family of  real numbers $\{\lambda_{k}\}^{m}_{k=1}$ such that
    \begin{equation}
        \|x_1-\sum\nolimits^{m}_{k=1} \lambda_k q_k \|\le
        \label{June28Equa6.1}
        \varepsilon/16.
    \end{equation}
    For each $1\le k\le m$,  by spectral theory for  $q_{k}x_2q_{k}$, there exist  an $n_k\in\mathbb N$, an orthogonal family of subprojections $\{q_{k,j}\}^{n_k}_{j=1}$ of $q_k$  with sum $q_k$ and a family of real  numbers $\{\eta_{k,j}\}^{n_k}_{j=1}$ such that $\|q_{k}x_2q_{k}- \sum\nolimits_{j=1}^{n_k}\eta_{k,j} q_{k,j} \|\le \varepsilon/16,$   in particular
    \begin{equation}
    \|q_{k,j}x_2q_{k,j}-  \eta_{k,j} q_{k,j} \|\le \varepsilon/16. \label{June28Equa6.2}
    \end{equation}
    Let $n=n_1+ \cdots + n_m$ and list $\{q_{k,j} : 1\le j\le n_k, 1\le k\le m\}$ as $\{p_k\}_{k=1}^n$ with
    $$\tau(p_1)\ge \tau(p_2)\ge \ldots \ge \tau(p_n). $$ By the inequalities (\ref{June28Equa6.1}) and (\ref{June28Equa6.2}), with a small perturbation, we can assume that there exist families of distinct  real numbers $\{\alpha_{k}\}^{n}_{k=1}$ and $\{\beta_{k}\}^{n}_{k=1}$  such that,  for $1\leq k \leq n$,
    \begin{enumerate}[label={\rm (\alph*)}]
        \item \label{prop6.5.1} $\alpha_k\ne 0$ ;
        \item \label{prop6.5.2} $\|x_1-\sum_k \alpha_k p_k \|\le \varepsilon/8;$
        \item \label{prop6.5.3} $\Vert p_k x_2 p_k -  \beta_k p_k   \Vert \leq   { \varepsilon}/{8}.$
    \end{enumerate}

    In terms of the condition that $ p\mathcal Mp\setminus\operatorname{Red}(p\mathcal Mp: p\mathcal Bp)\ne \emptyset$,   there exists a family of elements $\{a_{k}+i b_{k}\in p_k \M p_k \backslash \operatorname{Red}(p_k \mathcal M p_k: p_k \mathcal B p_k)\}^{n}_{k=1}$ with $a_k$ and $b_k$ being self-adjoint  such that
    \begin{enumerate}[resume,label={\rm (\alph*)}]
        \item  \label{prop6.5.4} $
            \Vert   a_k  \Vert \leq  { \varepsilon} /{8} \ \mbox{ and } \  \Vert  b_k  \Vert \leq   { \varepsilon}/ {8}, $ for $1\leq k \leq n$;
        \item  \label{prop6.5.5} for $1\leq j\neq k \leq n$,
        \begin{equation*}
        0\notin \sigma_{p_k \mathcal M p_k}(\alpha_k p_k+a_k) \text{ and }  \sigma_{p_j \mathcal M p_j}(\alpha_j p_j+a_j)\cap\sigma_{p_k \mathcal M p_k}(\alpha_k p_k+a_k)=\emptyset.
        \end{equation*}
    \end{enumerate}
    Note that $\mathcal{M}$ is a type ${\rm II}_1$ factor. For all $1\le i<j\le n$,  since $p_i\mathcal Mp_{j}\ne \{0\}$, we can choose $z_{ij}$ be an element in $p_i\mathcal Mp_{j}$   such that
      \begin{enumerate} [resume,label={\rm (\alph*)}]
        \item \label{Junly3prop6.5.6}  $\|z_{ij}-p_ix_2p_j\|\le \varepsilon/(8n^2)$;
         \item \label{Junly3prop6.5.7}  if $z_{ij}=v_{ij}h_{ij}$ is a polar decomposition of $z_{ij}$ in $\mathcal M$, then $v_{ij}^*v_{ij}=p_i$ and $h_{ij}$ is invertible in $p_j\mathcal Mp_j$.
    \end{enumerate}
    In fact, assume that $p_ix_2p_j=vh$ is a polar decomposition of $p_ix_2p_j$ in $\mathcal M$. From the assumption that $\tau(p_i)\ge \tau(p_j)$, there exists a partial isometry $u$ in $\mathcal M$ such that $u^*u=p_j-v^*v$ and $uu^*\le p_i-vv^*$. Let $f:\mathbb R\rightarrow \mathbb R$ be a function such that $f(t)=\varepsilon/(8n^2)$ for $t\le \varepsilon/(8n^2)$ and $f(t)=t$ for $t> \varepsilon/(8n^2)$. Then $vf(h)+(\varepsilon/(8n^2))u$ is an operator in $\mathcal M$ satisfying \ref{Junly3prop6.5.6} and \ref {Junly3prop6.5.7}.

    Define self-adjoint operators $y_1$ and $y_2$ in $\mathcal{M}$ in the following form
    \begin{equation*}
        y_1=\sum^{n}_{k=1}(\alpha_k p_k+a_k) \quad \mbox{ and } \quad y_2=\sum^{n}_{k=1}(\beta_k p_k+b_k)+\sum^{}_{ 1\leq k < \ell\leq n  } (z_{k \ell}+z_{k \ell}^*)
         .
    \end{equation*}
    From \ref{prop6.5.2}, \ref{prop6.5.3},  \ref{prop6.5.4} and \ref{Junly3prop6.5.6} , it follows that $\Vert (x_1 +i x_2)-(y_1+i y_2)\Vert\leq \varepsilon$.

    To complete the proof, we need only to show that $y_1+i y_2$ belongs to $\mathcal M\setminus \operatorname{Red}(\mathcal M: \mathcal B)$.
     Suppose that $q$ is a projection in $\B$ such that $[q,y_1+i y_2]=0$.  To prove $y_1+i y_2\in \mathcal M\setminus \operatorname{Red}(\mathcal M: \mathcal B)$, it suffices to prove that $q$ is in $\mathcal Z(\mathcal B)$, where $\mathcal Z(\mathcal B)$ is the center of $\mathcal B$.

    Now  \ref{prop6.5.5} and Lemma \ref{lemma5.2} entail that
   $
        p_1, \ldots, p_n, a_1, \ldots, a_n  \text { are in }  C^* (y_1,
        ).
    $ By computing  $p_ky_2p_k$ for $1\le k\le n$, we obtain that $  b_1, \ldots, b_n \text { are in }  C^* (y_1,
        y_2).$
    From $[q,y_1+iy_2]=0$, it follows that $q=q_1+\cdots+q_n$ and $q_k a_k=a_k q_k$, where  $q_k=p_kq$ is a sub-projection of $p_k$ for $1\leq k \leq n$. That $[q,y_1+i y_2]=0$ also implies that $q_k b_k=b_k q_k$ for $1\leq k \leq n$. Therefore, $[a_{k}+i b_{k},q_k]=0$.

    Since $a_{k}+i b_{k}\in p_k \M p_k \backslash \operatorname{Red}(p_k \mathcal M p_k: p_k \mathcal B p_k)$, the equality $[a_{k}+i b_{k},q_k]=0$ implies that $q_k\in \mathcal Z(p_k\mathcal B p_k)$. Thus, for each operator $b$ in $\mathcal B$, we know that
     \begin{equation}
       q_k (p_kbp_k)=   (p_kbp_k)q_k , \qquad \forall \ 1\le k\le n. \label{Sept14Equa6.3}
    \end{equation}

    For $1\le k< \ell \le n$, notice that $p_ky_2p_{\ell}\in  C^* (y_1,y_2)$. It follows that
    \begin{equation*}
        q(p_ky_2p_{\ell})= (p_k y_2 p_{\ell})q \quad \text{ or } \quad q_k  z_{k  \ell} =z_{k  \ell}q_{\ell}.
    \end{equation*}
    Now condition \ref{Junly3prop6.5.7} implies that, if $z_{k \ell}= v_{k \ell}h_{k \ell}$ is a polar decomposition of $z_{k \ell}$ in $\mathcal M$, then $v_{k \ell}^*v_{k \ell}=p_{\ell}$, $v_{k \ell}v_{k \ell}^*\le p_k$, and $h_{k \ell}$ is invertible in $p_{\ell}\mathcal M p_{\ell}$. Hence
    \begin{equation*}
        q_k v_{k \ell}h_{k \ell} =v_{k \ell}h_{k \ell} q_{\ell}  =v_{k \ell} q_{\ell}h_{k \ell}
    \end{equation*}
    Since $h_{k \ell}$ is invertible in $p_{\ell}\mathcal M p_{\ell}$, we conclude that $$ q_k v_{k \ell}  =v_{k \ell} q_{\ell}. $$
    For an operator $b$ in $\mathcal B$,  from the facts that $q_k\in\mathcal Z(p_k\mathcal B p_k)$, $q_{\ell}\in\mathcal Z(p_{\ell}\mathcal B p_{\ell})$, $v_{k \ell}^*v_{k \ell}= p_{\ell}$, and $ q_k v_{k \ell}  =v_{k \ell} q_{\ell}, $ we obtain that
    \begin{equation}
        p_kbp_{\ell} q_{\ell}=p_kb  (v_{k \ell}^*v_{k \ell}) q_{\ell}    = p_kb   v_{k \ell}^*  q_k v_{k \ell}  = q_k p_kb   v_{k \ell}^*v_{k \ell}    =q_k   p_k b p_{\ell}. \label{Sept14Equa6.4}
    \end{equation}

  By (\ref{Sept14Equa6.3}) and (\ref{Sept14Equa6.4}), for each operator $b$ in $\mathcal B$,we know that
   $$
   q b = bq,
   $$ 
   whence $q$ is in the center of $\mathcal B$. This ends the proof of the theorem.
\end{proof}

\begin{theorem}\label{main-goal}
    Let $\mathcal M$ be a non-$\Gamma$ type $ {\rm II}_1$ factor. Then, in the operator norm topology,  the set of reducible operators in $\mathcal M$ is nowhere dense and not closed  in $\mathcal M$.
\end{theorem}

\begin{proof}  
    It has been shown in Proposition \ref{June28Prop5.11} that, in the operator norm topology,  $\operatorname{Red}(\mathcal{M} )$ is not closed in $\mathcal{M}$. Next, we prove that $\operatorname{Red}(  \mathcal{M}  )$  is nowhere dense in $\mathcal M$.

    Let $\B = \ell^{\infty} ( \mathcal M)/c_0(\mathcal M)$. From Proposition \ref{may15Prop5.5},  we have
    \begin{equation*}
        \overline{\operatorname{Red}( \mathcal{M} )}^{\Vert\cdot\Vert}=\operatorname{Red}  ( \mathcal{M} : \mathcal{B} )
    \end{equation*}
    and, for any nonzero projection $p$ in $\mathcal M$, 
    \begin{equation*}
        \overline{\operatorname{Red}(p \mathcal{M} p)}^{\Vert\cdot\Vert} 
        = \operatorname{Red} (p \mathcal M p: p \mathcal B p ). 
    \end{equation*}
    By \Cref{lemma2.2},  $p \mathcal{M} p$ fails to have property $\Gamma$ for each nonzero projection $p$ in $\mathcal{M}$. From Theorem \ref{may14Thm5.14}, we have that $\overline{\operatorname{Red}(p \mathcal{M} p)}^{\Vert\cdot\Vert}\neq p \M p$. Equivalently, we have
    \begin{equation*}
          p \mathcal M p\setminus  \operatorname{Red} (p \mathcal M p:p\mathcal Bp )\ne \emptyset.
    \end{equation*}
    From Proposition \ref{N:B}, we obtain that $\M\setminus\operatorname{Red}  ( \mathcal M :\mathcal B )$ is dense in $\mathcal{M}$ in the operator norm topology. So $\M\setminus \overline{\operatorname{Red}( \mathcal{M} )}^{\Vert\cdot\Vert}$ is dense in $\mathcal{M}$ in the operator norm topology.  This guarantees that $\operatorname{Red}( \mathcal M)$ is a nowhere dense subset of $\mathcal{M}$ in the operator norm topology.
\end{proof}

\section*{Declarations}

\noindent{\bf Conflict of interest}
On behalf of all authors, the corresponding author states that there is no conflict of interest.

\end{document}